\documentclass[12pt]{amsart}
\usepackage{amssymb}
\usepackage{mathrsfs}
\usepackage{xspace}
\usepackage{enumerate}
\usepackage{cancel}
\usepackage{ifthen}
\usepackage[plain]{fancyref}
\usepackage{hyperref}
\usepackage[normalem]{ulem}
\usepackage{mathrsfs}
\usepackage{xcolor}
\usepackage[all]{xy} 
\CompileMatrices

\newtheorem{thm}{Theorem}[section]
\newtheorem{lem}[thm]{Lemma}
\newtheorem{cor}[thm]{Corollary}
\newtheorem{prop}[thm]{Proposition}

\newtheorem*{claim*}{Claim}

\theoremstyle{definition}
\newtheorem{defn}[thm]{Definition}

\newtheorem{question}[thm]{Question}
\newtheorem{conj}[thm]{Conjecture}

\theoremstyle{remark}
\newtheorem{rem}[thm]{Remark}

\newcommand*\fancyrefthmlabelprefix{thm}\frefformat{plain}{\fancyrefthmlabelprefix}{Theorem~#1}
\newcommand*\fancyreflemlabelprefix{lem}\frefformat{plain}{\fancyreflemlabelprefix}{Lemma~#1}
\newcommand*\fancyrefproplabelprefix{prop}\frefformat{plain}{\fancyrefproplabelprefix}{Proposition~#1}
\newcommand*\fancyrefcorlabelprefix{cor}\frefformat{plain}{\fancyrefcorlabelprefix}{Corollary~#1}
\newcommand*\fancyrefclaimlabelprefix{claim}\frefformat{plain}{\fancyrefclaimlabelprefix}{Claim~#1}
\newcommand*\fancyreffactlabelprefix{fact}\frefformat{plain}{\fancyreffactlabelprefix}{Fact~#1}
\newcommand*\fancyrefquestionlabelprefix{question}\frefformat{plain}{\fancyrefquestionlabelprefix}{Question~#1}
\newcommand*\fancyrefconjlabelprefix{conj}\frefformat{plain}{\fancyrefconjlabelprefix}{Conjecture~#1}
\newcommand*\fancyrefdefnlabelprefix{defn}\frefformat{plain}{\fancyrefdefnlabelprefix}{Definition~#1}
\newcommand*\fancyrefconstlabelprefix{const}\frefformat{plain}{\fancyrefconstlabelprefix}{Construction~#1}
\newcommand*\fancyrefsetuplabelprefix{setup}\frefformat{plain}{\fancyrefsetuplabelprefix}{Setup~#1}
\newcommand*\fancyrefexlabelprefix{ex}\frefformat{plain}{\fancyrefexlabelprefix}{Example~#1}
\newcommand*\fancyrefremlabelprefix{rem}\frefformat{plain}{\fancyrefremlabelprefix}{Remark~#1}
\frefformat{plain}{\fancyrefeqlabelprefix}{(#1)}
\newcommand*\fancyrefitemlabelprefix{item}\frefformat{plain}{\fancyrefitemlabelprefix}{(#1)}

\def\repeat#1#2 {\expandafter\gdef\csname B#1\endcsname {\mathbb{#1}}
  \ifthenelse{\equal{#2}{*}}{}{\repeat #2 }}
\repeat ABCDEFGHIJKLMNOPQRSTUVWXYZ*
\def\repeat#1#2 {\expandafter\gdef\csname rm#1\endcsname {\mathrm{#1}}
	\ifthenelse{\equal{#2}{*}}{}{\repeat #2 }}
\repeat ABCDEFGHIJKLMNOPQRSTUVWXYZ*
\def\repeat#1#2 {\expandafter\gdef\csname C#1\endcsname {\mathcal{#1}}
  \ifthenelse{\equal{#2}{*}}{}{\repeat #2 }}
\repeat ABCDEFGHIJKLMNOPQRSTUVWXYZ*
\def\repeat#1#2 {\expandafter\gdef\csname bf#1\endcsname {\boldsymbol{#1}}
  \ifthenelse{\equal{#2}{*}}{}{\repeat #2 }}
\repeat abcdefghijklmnopqrstuvwxyzABCDEFGHIJKLMNOPQRSTUVWXYZ*
\def\repeat#1#2 {\expandafter\gdef\csname scr#1\endcsname {\mathscr{#1}}
  \ifthenelse{\equal{#2}{*}}{}{\repeat #2 }}
\repeat abcdABCDEFGHIJKLMNOPQRSTUVWXYZ*

\newcounter{last-index}

\long\def\nonsection#1\section#2{\section{#2}}
\let\oldsection\section

\let\epsilon\varepsilon  
\let\isom\cong
\let\semidirect\rtimes
\let\normal\triangleleft

\newcommand\Orientation[1][]{
  \ifthenelse{\equal{#1}{}}
  {\CO_{p}}
  {\CO_{p,#1}}}

\DeclareMathOperator{\res}{res}
\DeclareMathOperator{\Diff}{Diff}
\DeclareMathOperator{\Cr}{Cr}
\DeclareMathOperator{\Bir}{Bir}
\DeclareMathOperator{\Hom}{Hom}
\DeclareMathOperator{\Tor}{Tor}
\newcommand{\HomSheaf}{\ensuremath{\underline{\CH\kern-1pt om}}}
\DeclareMathOperator{\Aut}{Aut}
\DeclareMathOperator{\Gal}{Gal}
\DeclareMathOperator{\GL}{GL}
\DeclareMathOperator{\SL}{SL}
\DeclareMathOperator{\PSL}{PSL}
\DeclareMathOperator{\im}{im}
\DeclareMathOperator{\rank}{rk}
\DeclareMathOperator{\Stab}{Stab}
\newcommand{\Fix}[2]{{#1}^{#2}}
\DeclareMathOperator{\Aff}{Aff}
\DeclareMathOperator{\Homeo}{Homeo}

\begin{document}
	
\title[Finite subgroups of homeomorphism groups]{Finite subgroups of
  the homeomorphism group of a compact topological manifold are almost nilpotent}
\author{Bal\'azs Csik\'os \and L\'aszl\'o Pyber \and Endre Szab\'o}
\maketitle

\begin{abstract}
  Around twenty years ago Ghys conjectured that finite subgroups of
  the diffeomorphism group of a
  compact smooth manifold $M$ have an abelian normal subgroup of index at
  most $a(M)$, where $a(M)$ depends only on $M$. First we construct a family of counterexamples to
  this conjecture including, for example, the product space $T^2 \times S^2$.
 
  Following the first appearance of our counterexample on the arXiv 
  Ghys put forward a revised conjecture, which predicts only the existence of a nilpotent normal subgroup of index at
  most $n(M)$.  Our main result is the proof of the revised Ghys
  conjecture. More generally, we show that the same result
  holds for homeomorphism groups of not necessarily compact topological manifolds with finitely generated homology groups.  

  Our proofs are based on finite group theoretic results which provide a general strategy for proving similar Jordan-type theorems.
\end{abstract}

\footnote{
We would like to dedicate this work to our teacher, Gabriella Thiry
(Gabi néni).
}

\oldsection{Introduction}
The following is a summary of two classical results due to
H.~Min\-kowski \cite{Minkowski}
and C. Jordan \cite{jordan1877memoire}.
See \cite{Guralnick_Lorenz} and \cite{breuillard2011Jordan}
for a modern rendering of Minkowski's and Jordan's original arguments.

\begin{thm}[Minkowski, Jordan]
  \label{thm:Minkowski-Jordan}
  If $K$ is a number field, then there is a constant $B=B(d,K)$
  such that for any finite group $G<\GL(d,K)$, one has $|G|\le B$.
  If $K$ is an arbitrary field of characteristic zero, then there is a
  constant $J=J(n)$ such that for any finite group \break
  $G<\GL(d,K)$, 
  there exists an abelian normal subgroup $A\le G$ of index at most $J$.
\end{thm}

Let us say that a group $\CG$ has the \emph{Jordan property} if there
is a constant $a(\CG)$ such that any finite subgroup $G$ of $\CG$ has an
abelian normal subgroup of index at most $a(\CG)$. More informally,
an infinite group $\CG$ has the Jordan property if all of its finite
subgroups are ``almost'' abelian. This terminology was introduced a few
years ago by Popov \cite{popov2010makar}.
Jordan's theorem states that the groups $\GL(d,\BC)$ have the Jordan property.
Boothby and Wang
\cite{boothbyWilliamWhang1964ConnectedLieGroupsAreJordan}
proved that connected Lie groups have
the Jordan property (see also \cite{lee1976torsion-subgroups-Lie-groups}).
Moreover, connected algebraic groups have
the Jordan property \cite{Meng-Sheng-Zhang2015jordan}.

Around twenty years ago \'Etienne Ghys
conjectured in several lectures (see  \cite{Ghys}) that for any
compact smooth manifold $M$,
there is a constant $a(M)$ such that any finite group acting smoothly and
effectively on $M$ has an abelian normal subgroup of index at
most $a(M)$, that is, the diffeomorphism group of $M$ has the Jordan
property. This conjecture first appeared in print in a 2011 survey of
Fisher on the Zimmer program \cite{Fisher}.

The particular case in which $M$ is a
sphere was also independently asked in several talks by Walter Feit
and later by Bruno Zimmermann. A positive answer to the
Feit-Zimmermann question was recently obtained by Mundet i Riera
\cite{Riera_spheres}.
Zimmermann \cite{Zimmermann} has shown
using Thurston's geometrization conjecture
(proved by Perelman) that the Ghys
conjecture holds for compact $3$-manifolds. Mundet i Riera has confirmed the Ghys conjecture in several
other cases: for tori \cite{Riera_torus}, acyclic manifolds, homology spheres and
manifolds with nonzero Euler characteristic \cite{Riera_spheres}.

Despite the rising expectations,
the authors of the present paper proved in \cite{CsPySz}
that the conjecture is false:
the diffeomorphism group of $T^2\times S^2$
does not have the Jordan property.
Our counter-example was inspired by an algebro-geometric construction
of Yu.~G.~Zarhin \cite{Zarhin},
who proved that the birational automorphism group of $E\times\BP^1$,
where $E$ is an elliptic curve, does not have the Jordan property.

In view of the above results, Mundet i Riera \cite{Riera_spheres}
posed the intriguing problem to
characterize compact smooth manifolds for which Ghys' conjecture is true.

Following the appearance of our counterexample \cite{CsPySz},
Ghys revised his conjecture.
In a talk \cite{Ghys-lecture}, he suggested the following.

\begin{conj} [Ghys]
  Let $M$ be a compact smooth manifold. Then any finite
  subgroup of $\Diff(M)$ has a nilpotent normal subgroup of index at most
  $n(M)$.
\end{conj} 

Note that Jordan's theorem can be derived as a corollary of the
classical Zassenhaus lemma, which states  that there exists a neighborhood $\Omega$ of the identity in $\GL(d,\BR)$ such that if $\Gamma$ is a discrete
subgroup of $\GL(d,\BR)$, then $\Gamma\cap\Omega$ generates a nilpotent subgroup. Hence the revised Ghys conjecture is related to earlier work of Ghys aiming at obtaining
analogues of the Zassenhaus lemma for analytic diffeomorphism groups
(see \cite{Fisher}).

It was suggested by Fisher \cite{Fisher} that for the original question of Ghys,
it would be more natural to ask about finite groups of homeomorphisms
and not assume the differentiability of the maps.

Our main result is a proof of the revised Ghys conjecture. Actually, we
prove the following stronger result.

\begin{thm}\label{thm:Ghys-conjecture-nilpotent}
  Let $M$ be a compact topological manifold.
  If a finite group $G$ acts continuously and effectively on $M$,
  then it has a nilpotent normal subgroup $N\normal G$
  of index at most $n(M)$,
  where $n(M)$ depends only on the homotopy type of $M$.
\end{thm}

For our proof to work it is essential to allow non-compact manifolds
as well.
Therefore we prove the following more general result,
which implies \fref{thm:Ghys-conjecture-nilpotent} immediately.

\begin{thm}\label{thm:Ghys-conjecture-open-manifolds}
  Let $M$ be a (possibly non-compact) topological
  manifold whose homology group $H_*(M;\BZ)$ is finitely generated (as an abelian group).
  If a finite group $G$ acts continuously and effectively on $M$,
  then it has a nilpotent normal subgroup $N\normal G$
  of index at most $n\big(\dim(M),\allowbreak H_*(M;\BZ)\big)$,
  where $n\big(\dim(M),H_*(M;\BZ)\big)$ depends
  only on $\dim(M)$ and $H_*(M;\BZ)$.
\end{thm}

In our counterexamples showing that the diffeomorphism group of a
compact smooth manifold $M$ may not have the Jordan property (see \fref{sec:counter-example}), the finite
groups acting on $M$ are $2$-step nilpotent as are the groups in numerous further
examples obtained by Mundet i Riera \cite{Riera2016symplectic} and
D\'avid R. Szab\'o \cite{Szabo} (see also \cite{Szabo-David-Thesis})
based on our construction. This raises the question whether \fref{thm:Ghys-conjecture-nilpotent} can be strengthened as follows.
\begin{question}
  \label{question:Ghys-with-2-step-nilpotent}
  Let $M$ be a compact smooth manifold. Is there a bound $n'(M)$ such that any finite
  subgroup of $\Diff(M)$ has a 2-step nilpotent normal subgroup of index at most
  $n'(M)$?
\end{question} 
In view of the results of the Ph.D. thesis of
D\'avid R. Szab\'o~\cite{Szabo-David-Thesis} an
affirmative answer would be best possible.

An affirmative answer to the above question was proved by Mundet i Riera and S\'aez-Calvo \cite{Riera_4d} in dimension $4$. Their proof exploits the fact that in dimension $4$, the class of freely acting finite subgroups of $\Diff(M)$ do have the Jordan property, therefore if a subgroup is far from abelian, then it  must have elements with non-empty fixed point set. We remark that this approach is limited to dimension $4$, as we shall see in  \fref{sec:counter-example} that the Heisenberg type groups justifying the failure of the Jordan property of the diffeomorphism group of $T^2\times S^3$ act freely on the product $T^2\times S^3$.

\smallskip

The question whether groups related to some geometric structure
have some variant of the Jordan property
has been asked in other contexts as well.  Serre \cite{serre2008Bourbaki-Cremona}
proved that the plane Cremona group
$\Cr_2(\BC)$, the group of birational automorphisms of 
the projective plane $\BC\BP^2$,
has the Jordan property.
He asked if this also holds for the higher rank
Cremona groups $\Cr_n(\BC)$.

Serre's question was
answered positively for the Cremona group $\Cr_3(\BC)$
by Prokhorov and Shramov \cite{Prokhorov-Shramov-Cremona}.
More generally, they have shown that if the BAB
(Borisov-Alexeev-Borisov) conjecture holds, then all Cremona groups
have the Jordan property. The BAB conjecture has been
confirmed by Birkar \cite{birkar2016Fanos-are-bounded}.
This, in particular, completes the proof of a
positive answer to the question of Serre.

Popov \cite{popov2014jordan}
extended Serre's question to birational automorphism groups of general varieties. 
He proved that the birational automorphism groups of complex algebraic
surfaces have the Jordan property, except for surfaces birationally
equivalent to one of the surfaces
$E\times\BP^1$, where $E$ is any elliptic curve.
Zarhin \cite{Zarhin} has shown, that $\Bir(E\times\BP^1)$ does not
have the Jordan property. 
 Later Prokhorov and Shramov \cite{Prokhorov-Shramov-Bir-X}
proved, that for any complex algebraic variety $X$,
all finite subgroups of $\Bir(X)$
have a soluble normal subgroup of bounded index.
Using their ideas, Guld \cite{Guld} has shown that actually finite
subgroups of $\Bir(X)$ have a nilpotent normal subgroup of bounded
index and nilpotency class at most two.
We consider this as a further indication that
\fref{question:Ghys-with-2-step-nilpotent}
is the right question to ask.

In a recent paper, Mundet i Riera \cite{Riera2016symplectic}
obtained similar results for the group of
symplectomorphisms of a compact symplectic manifold. In yet another
recent paper, Prokhorov and Shramov \cite{prokhorov2016cremona-3fold}
have classified algebraic threefolds for
which the group of birational automorphisms does not have the Jordan
property.

It turns out (see \fref{thm:Mann-Su_v0}) that the finite groups we
encounter have bounded rank in the following sense.

\begin{defn}
  The \emph{rank} of a finite group $G$ is the
  minimal integer $r$ such that every subgroup $H$ of $G$ is
  $r$-generated.  
\end{defn}

Various group-theoretic properties, like being abelian-by-bounded
or nilpotent-by-bounded can be established for finite groups $G$ of
bounded rank by proving that the same property holds for subgroups of $G$
with a very simple structure. This somewhat surprising principle has
been found independently by Mundet i Riera and Turull \cite{Riera_Turull}
and the authors of the present paper (see a remark concerning this in
\cite{Riera_Turull}). The results
of \cite{Riera_Turull} have found applications much earlier in
\cite{Riera_spheres}.
We waited with the publication of our variant till we could complete
the proof of the applications described here.

An essential tool in the course of proving
\fref{thm:Ghys-conjecture-open-manifolds}
is the following group theoretic reduction theorem,
which is our variant of the above principle.
This reduction allows us
to consider only very special type of groups in the geometrical arguments.
\begin{thm} [Reduction Theorem]
	\label{thm:Nilpotent-or-Special-by-cyclic_v0}
	Let $G$ be a finite group of rank
	$r$. Then for all $T > 0$, there is an integer $I(r,T) > 0$ such that
	one of the following holds.
	\begin{enumerate}[(a)]
		\item \label{item:1}
		$G$ has a nilpotent normal subgroup of index at most $I(r,T)$.
		\item \label{item:2}
		For some distinct primes $p$ and $q$ the group $G$ has a
		subgroup of the form $P\semidirect \BZ_m$,
		where $P$ is a special $p$-group, $m$ is a
		power of $q$, and the image of $\BZ_m$ in $\Aut(P)$ has order at
		least $T$.
		In particular, we have $|P|\le p^{2r}$.
	\end{enumerate}
\end{thm}

A similar reduction theorem for checking the Jordan property was
obtained earlier by the present authors, and independently by Mundet i Riera and Turull \cite{Riera_Turull}.

\fref{thm:Nilpotent-or-Special-by-cyclic_v0} will be proved in  \fref{sec:jordan-type-theorems} as  \fref{thm:Nilpotent-or-Special-by-cyclic}. Both the proofs of \fref{thm:Mann-Su_v0} and \fref{thm:Nilpotent-or-Special-by-cyclic_v0} rely on group theoretic techniques, in particular, on the
Classification of Finite Simple Groups (CFSG). It is an interesting
question whether the use of CFSG can be avoided.
The generality of our statement suggests that, perhaps,
group theoretic arguments can be completely eliminated.
Moreover, the analogous algebro-geometric results
do not use CFSG at all.
This leads to the following:

\begin{question}
	Is there a geometric proof of \fref{thm:Ghys-conjecture-nilpotent}?
\end{question}

\subsection*{Scheme of the proof of	\fref{thm:Ghys-conjecture-open-manifolds}}
\label{sec:sketch-proof}

The starting point is to establish an upper bound $r$ on the rank of the
group $G$. For compact manifolds and abelian groups, it is a classical result of Mann-Su
\cite{Mann_Su}, our more general situation is tackled in \cite{CMPS}
(see \fref{thm:Mann-Su_v0} below).

The rank bound and \fref{thm:Nilpotent-or-Special-by-cyclic_v0}
imply, with a little bit of thought,
that it is enough to prove the theorem for groups of the form
$G\isom P\semidirect \BZ_m$
where $P$ is a $p$-group of size $|P|\le p^{2r}$ for some prime $p$,
and $m$ is not divisible by $p$.
So from now on $G$ is such a group.

Minkowski's bound (\fref{thm:Minkowski-Jordan}) gives us a subgroup
$H\le G$ of bounded index which acts trivially on $H^*(M;\BZ)$.
By the universal coefficient theorem, $H$ acts trivially on
$H_*(M;\BF_p)$ as well.
We replace $G$ with $H$, so from now on  $G$ acts trivially on $H_*(M;\BF_p)$.

Let $U\subset M$ be the largest open subset where the $P$-action is
free.
By a result of \cite{CMPS}, $\dim H_*(U;\BZ_p)$ is bounded from above,
and it is tempting to replace $M$ with $U$ to reduce our statement to
free actions, but  \cite{CMPS} gives us no control over the
$\BZ_m$-action on $H_*(U;\BZ_p)$.
In \fref{sec:Reduction-to-free-action} we study this situation.
The Borel Fixed Point Formula (see \fref{prop:Borel-fixpoint-formula})
gives us points whose stabilizer subgroups in $P$ are isomorphic to
$\BZ_p$. The main result of \cite{CMPS} implies that these cyclic
subgroups are $\BZ_{m'}$-invariant for some $\BZ_{m'}\le\BZ_m$ of
bounded index.
After a careful analysis of
the fixed point submanifolds of these cyclic subgroups
and the corresponding  
$\BZ_p\semidirect\BZ_{m'}$-actions around them,
we can safely remove these submanifolds from $M$.
Repeating this step inductively,
we arrive at $U$, and we deduce that
there is a subgroup $\BZ_{m''}\le\BZ_m$ of bounded index
which acts trivially on $H_*(U;\BZ_p)$
(see \fref{lem:action-of-H-on-cohomology-of-U-is-trivial}).
Finally, we replace $M$ with $U$,
and reduce the statement to the case of free action of $P$.

Note, that in the above reductions we lost control of the integer homology,
we only have a bound on the mod $p$ homology of $M$.
Moreover, even if we started with a compact manifold, we lost
compactness along the way.
So it is essential for our proof to generalize the conjecture of Ghys
and include open manifolds. 

So we are reduced to the following scenario. 
$G=P\semidirect \BZ_m$ as above, $G$ acts on
a topological manifold $M$ such that $\dim H_*(M;\BZ_p)=B$ is bounded, and $P$ acts freely.
Let $\BZ_{m'}$ be the image of the conjugacy action of $\BZ_m$ on $P$.
Then $G$ has a nilpotent subgroup of index $m'$,
namely $P\semidirect\BZ_{m/m'}=P\times\BZ_{m/m'}$,
so it is enough to bound $m'$.

\fref{lem:elementary-p-by-cyclic-free-action} tackles this for the
case when $P$ is an elementary abelian $p$-group.
The idea is the following.
It is more convenient to use Poincar\'e duality and
switch to cohomology.
We study the action of $\BZ_{m'}$
on the Borel spectral sequence calculating the cohomology of $M/P$
(see \fref{prop:Leray-spectral-sequence}).
For each $i,j$, one can easily determine the irreducible
representations of $\BZ_{m'}$
which show up in $\BE_2^{i,j}$. On the other hand,
$\BE_\infty^{i,j}=0$ for $i+j>n$, so a lot of cancellations must occur.
But an irreducible representation can be cancelled only by the same
representation appearing in certain other $\BE_2^{k,l}$ terms with
$k+l=i+j-1$.
Using general representation theory we find a bound $N$
such that all irreducible representations of $\BZ_{m'}$ must occur in
$\bigoplus_{i+j\le N}\BE_2^{i,j}$. This gives a bound on $m'$,
and proves the theorem for this case.

Finally, in the general case,
when $P$ is an arbitrary $p$-group of order $p^\rho$
for some $\rho\le2r$
(see \fref{lem:nilpotent-by-cyclic-free-action} for the precise statement)
we use induction on $\rho$.
Let $E$ be the subgroup consisting of the elements of order at
most $p$ in the center of $P$. This is a characteristic subgroup of
$P$, hence $\BZ_m$ normalizes it.
First we apply the already known case to the subgroup
$E\semidirect\BZ_m\le G$,
which gives us a subgroup $\BZ_{m_1}\le\BZ_m$ of bounded index
which commutes with $E$.
So $E$ is in the center of the subgroup $G_1=P\semidirect\BZ_{m_1}\le G$.
Then we apply the induction hypothesis to the $G_1/E$ action on the
manifold $M/E$,
and find a subgroup $\BZ_{m_2}\le\BZ_{m_1}$ of bounded index which
commutes with $P/E$.
Since $m_2$ is not divisible by $p$,
this $\BZ_{m_2}$ must commute with $P$ as well.
Therefore $P\semidirect\BZ_{m_2}=P\times\BZ_{m_2}$
is nilpotent.
This completes the induction.

\subsection*{Acknowledgements}
L.~Pyber and E. Szab\'o were supported by the National Research, Development and
Innovation Office (NKFIH) Grant K138596. E.~Szab\'o was also supported
by the NKFIH Grant K120697. B.~Csik\'os was supported by the NKFIH Grant K128862.
The project leading to this application has received funding from the European Research Council (ERC) under the European Union's Horizon 2020 research and innovation programme (grant agreement No 741420).

The authors are indebted to Andr\'as N\'emethi and Andr\'as Stipsicz for helpful and fruitful discussions.

\section{Counterexamples}
\label{sec:counter-example}

In this section, we construct a family of compact real analytic manifolds, the real analytic diffeomorphism groups of which do not have the Jordan property. These manifolds will be the total spaces of bundles over the $2$-dimensional torus $T^2=\BR^2/\BZ^2$ with compact fibers and structure group $G$, where $G$ is a compact connected Lie group with finite center acting effectively on the fibers. It will be shown that infinitely many of the Heisenberg type groups {$G_n=(\mathbb Z_n\times \mathbb Z_n)\ltimes \mathbb Z_n$} with multiplication rule
\begin{equation}\label{eq:Heisenberg}
(a,b,c)\cdot(a',b',c')=(a+a',b+b', c+c'+ab')
\end{equation}
have an effective real analytic action on such a manifold, while the index of an abelian subgroup of $G_n$ is at least $n$. 

The simplest examples of this type, the total spaces of the oriented $S^2$-bundles over $T^2$,   appeared first in the preprint of the authors \cite{CsPySz} motivated by the ideas of Zarhin \cite{Zarhin}. We remark that in the special case of the trivial $S^3$ bundle with structure group $\mathrm{SU}(2)$, our construction  provides a \emph{free} action of all the groups $G_n$ on the product manifold $T^2\times S^3$. This shows that the diffeomorphism group of $T^2\times S^3$ does not even have the weakened Jordan property considering only freely acting finite groups of diffeomorphisms. 
	
\subsection{Principal bundles over the torus}\label{subsec:bundles_over_torus}
There is a simple way to classify real analytic principal bundles over the torus $T^2$.
\begin{lem}\label{lem:bundle_classification}
	Let $G$ be a connected compact Lie group. Then the isomorphism classes of real analytic principal $G$-bundles over the torus $T^2$ are classified by the elements of the fundamental group $\pi_1(G)$.
\end{lem}\label{lem:induced_bundle_isomorphism} 
\begin{proof} Real analytic and topological classification of principal $G$-bundles coincide (see \cite{Guaraldo}), so we focus on the latter one. Principal $G$-bundles over the torus are classified by the homotopy classes of maps $T^2\to BG$ into a classifying space $BG$ of principal $G$-bundles. Consider the usual CW structure of $T^2$ with a single $2$-cell glued to the $1$-skeleton $T^2_{(1)}$ equal to a bouquet of two circles. Since $G$ is connected, $BG$ is simply connected, and $[T^2,BG]=[T^2/T^2_{(1)},BG]=\pi_2(BG)\cong \pi_1(G)$. 
\end{proof}

 If $H$ is a Lie group, $\xi=(E\to B)$ is a principal $H$-bundle, then for any smooth manifold $F$ equipped with a left $H$-action $\theta$, there is a fiber bundle $\xi_{\theta}[F]=(E\times_HF\to B)$ associated to the principal bundle $\xi$, having fibers diffeomorphic to $F$.  The correspondence between the principal $G$-bundles over $T^2$ and $\pi_1(G)$ is canonical in the following sense. If $\phi\colon H\to G$ is a smooth homomorphism from $H$ into the Lie group $G$, then $H$ acts on $G$ by left translations via the homomorphism $\phi$, and the associated bundle $\xi_{\phi}[G]$ is a principal $G$-bundle. 
 In the case when the base space of $\xi$ is $T^2$, the element of $\pi_1(G)$ corresponding to $\xi_{\phi}[G]$ is the image of the element of $\pi_1(H)$ corresponding to $\xi$ under the induced homomorphism $\phi_*\colon \pi_1(H)\to \pi_1(G)$.
 
 As a special case, principal $S^1$-bundles over the torus $T^2$ correspond to elements of $\pi_1(S^1)\cong \BZ$. The isomorphism with $\BZ$ can be given explicitly by fixing an orientation of the torus and assigning to each circle bundle $\xi$ the integral of its first Chern class $c_1(\xi)$ over the torus. We shall denote by $\xi^n$ the principal $S^1$-bundle over $T^2$ with Chern number $n$.

If $G$ is a compact connected Lie group with finite center, then $\pi_1(G)$ is finite, so there are only a finite number of isomorphism classes of principal $G$-bundles over $T^2$. In particular, if $G$ is simply connected, e.g., $G=\mathrm{SU}(n)$, then all the $G$ bundles over $T^2$ are trivial. Since every element of $\pi_1(G)$ can be represented by a power of an \emph{injective} group homomorphism $\phi\colon S^1\to G$, any principal $G$-bundle over $T^2$ is isomorphic to the bundle $\xi^1_{\phi^k}[G]\cong \xi^k_{\phi}f[G]$ for a suitably chosen injective homomorphism $\phi$ and integer $k$. If $\phi$ represents an element of order $m$ in $\pi_1(G)$, then $\xi^k_{\phi}[G]$ and $\xi^l_{\phi}[G]$ are isomorphic principal $G$ bundles if and only if $m$ divides $k-l$.

\subsection{Weak automorphisms of bundles over the torus}

Recall that by a \emph{weak isomorphism} between the smooth bundles $\xi=(E \xrightarrow{\pi} B)$ and $\xi'=(E' \xrightarrow{\pi'}B')$ with the same fiber type we mean a pair of diffeomorphisms $\Phi\colon E\to E'$ and $\bar \Phi\colon B\to B'$ such that $\bar \Phi\circ \pi=\pi'\circ \Phi$ and $\Phi$ is a structure preserving isomorphism between the fibers. 

A weak isomorphism is uniquely defined by the map $\Phi$. Consequently, the group of weak automorphisms of a bundle can be embedded into the diffeomorphism group of the total space as a subgroup. 

\begin{lem}\label{lem:G_k_action} For any integer $n\ge 2$, the Heisenberg type group $G_n=(\mathbb Z_n\times \mathbb Z_n)\ltimes \mathbb Z_n$ with multiplication rule \eqref{eq:Heisenberg} acts effectively by real analytic weak automorphisms  on a real analytic model of the principal $S^1$-bundle $\xi^n$. 
\end{lem}
\begin{proof}
We construct a real analytic model of $\xi^n$ for $n\in \mathbb Z$ as follows. Let $\Xi=(\mathbb R^2\times S^1\to \mathbb R^2)$ be the trivial principal $S^1$-bundle over the plane. For a given  $n\in \mathbb Z$, consider the  action of $\mathbb Z^2$ on $\Xi$ by weak automorphisms given by the action $\Phi_n\colon \BZ^2\times (\BR^2\times S^1)\to\BR^2\times S^1 $, 
\[
\Phi_n\big((k,l),(x,y,z)\big)=(x+k,y+l,e^{2\pi i \,kn y}z)
\]
on the total space, where $(k,l)\in \mathbb Z^2$, $(x,y)\in \mathbb R^2$, and  $z\in S^1\subset \BC$.
 As $\Xi$ and the action $\Phi_n$ are real analytic, factoring $\Xi$ with this action, we obtain a real analytic model $E_n\xrightarrow{\pi_n}T^2$ for the principal $S^1$-bundle $\xi^n$. Denote the $\Phi_n(\mathbb Z^2)$-orbit of 
$(x,y,z)\in \mathbb R^2\times S^1$ by $\langle x,y,z\rangle_n\in E_n$. 

Take the ring of modulo $n$ residue classes of integers as a model for $\BZ_n$, and denote by $[k]_n\in \BZ_n$ the residue class of $k\in \BZ$. Then the demanded action of $G_n$ on the bundle $\xi^n$ is induced by the action $\Psi_n\colon G_n\times E_n\to E_n$,
\[                                                               
\Psi_n\big(([k]_n,[l]_n,[m]_n), \langle x,y,z\rangle_n\big) = \langle x+k/n,y+l/n,e^{2\pi i(k y+\frac{m}{n})}z\rangle_n 
\]
on the total space of $\xi^n$.
\end{proof}

\begin{lem}
  If $A<G_n$ is a commutative subgroup, then $|G_n:A|\geq n$.
\end{lem}
\begin{proof}
  See \cite[Section 3.]{Zarhin}.
\end{proof}

\begin{thm}
Let $G$ be a connected compact Lie group with finite center, $F$ be a compact real analytic manifold with an effective real analytic left $G$-action $\theta$. Then for any principal $G$-bundle $\eta$ over the torus $T^2$, the weak automorphism group of the associated bundle $\eta[{}_{\theta}F]$ contains infinitely many of the groups $G_n$ as a subgroup. In particular, the diffeomorphism groups of the total spaces of these bundles do not have the Jordan property.
\end{thm}

\begin{proof} As it was observed at the end of subsection \ref{subsec:bundles_over_torus}, we can find an injective homomorphism $\phi\colon S^1\to G$ such that $\eta\cong\xi^k_{\phi}[G]$ for some $k\in \BZ$, and if $\eta$ corresponds to an element of order $m$ in $\pi_1(G)$ by the statement of \fref{lem:bundle_classification}, then  $\eta\cong\xi^l_{\phi}[G]$ whenever $k\equiv l$ mod $m$. By \fref{lem:G_k_action} and the naturality of the associated bundle construction, if the integer $l\ge 2$ is congruent to $k$ modulo $m$, then $G_l$ acts effectively on $\xi^l$ by weak automorphisms, and this action induces an action of $G_l$ also on the associated bundles $\eta\cong\xi^l_{\phi}[G]$  and  $\eta_{\theta}[F]\cong(\xi^l_{\phi}[G])_{\theta}[F]$. Since $\phi$ is injective, $\theta$ is effective, the resulting $G_l$-actions on $\eta_{\theta}[F]$ and its total space are effective. Representing these bundles with their real analytic models described above, all these actions will be real analytic. 
\end{proof}

\section{The group theoretic reduction theorem}
\label{sec:jordan-type-theorems}

In this section, we prove our key group-theoretic tool,
\fref{thm:Nilpotent-or-Special-by-cyclic_v0}
and other related results. These results can be used to show that,
under certain conditions, a finite group contains an abelian or
nilpotent normal subgroup of small index.

We will start with an account of the theory of finite groups of
bounded rank. Recall that the \emph{rank} of a finite group $G$ is the
minimal integer $r$ such that every subgroup $H$ of $G$ is
$r$-generated.  

The following lemma is proved in \cite[Corollary~1.8]{p-group-stuff}.
\begin{lem}
  \label{lem:rank-and-elementary-abelian-subgroups}
  If every elementary abelian subgroup of a finite group $G$ has rank at most $d$,
  then $G$ has rank at most $\frac12d^2+2d+1$.
\end{lem}

It is an easy observation that a finite abelian subgroup of $\GL(n,\BC)$ has
rank at most $n$. Moreover, as shown in \cite{Kovacs_Robinson}, the rank of any
finite subgroup of $\GL(n,\BC)$ is at most $[3n/2]$.  Groups of bounded
rank also occur naturally in the study of finite groups acting on
compact topological manifolds. Namely, as mentioned in the introduction, by a result of Mann and Su \cite{Mann_Su},
the ranks of elementary abelian groups acting faithfully by homeomorphisms on a given
compact topological manifold $M$ are bounded. This implies that actually there is
a bound on the ranks of finite groups acting on $M$ by homeomorphisms.

Shalev \cite{Shalev} 
has given a description of the structure of finite groups of bounded
rank.
\begin{thm}
  \label{thm:2-1}
  Let $G$ be a finite group of rank $r$.
  Then there exists a chain $1\normal G_2\normal G_1\normal G$ of
  characteristic subgroups of $G$ such that
  \begin{enumerate}[\indent(a)]
  \item $|G/G_1|$ is bounded in terms of $r$;
  \item $|G_1/G_2|\isom S_1\times\dots \times S_k$, where
    $0\le k\le r$, and for $1\le i\le k$, the group
    $S_i=X_{n_i}(p_i^{e_i})$ is a simple group of Lie type, such that
    $n_i$ and $e_i$ are bounded in terms of $r$;
  \item $G_2$ is soluble.
  \end{enumerate}
\end{thm}

Note that the above structure theorem relies on the Classification
of Finite Simple Groups (CFSG) in an essential way. It would
be most interesting to prove such a result without CFSG.
This would
be a natural analogue of the CFSG-free description of the structure of
finite linear groups of bounded dimension over arbitrary fields. That
structure theorem was first obtained by Weisfeiler
\cite{weisfeiler1984post}
using CFSG and later Larsen and Pink \cite{Larsen-Pink}
gave an amazing ``elementary'' proof.

We will use \fref{thm:2-1} to reduce the proof of \fref{thm:Nilpotent-or-Special-by-cyclic_v0}, the
main result of this section, to the soluble case. We prove the
following.

\begin{prop}
  \label{prop:2-2}
  Let $G$ be a finite group of rank $r$ and
  assume that all soluble subgroups $S$ of $G$ contain a nilpotent normal
  subgroup of index at most $t$. Then $G$ contains a nilpotent normal
  subgroup of index bounded in terms of $r$ and $t$.
\end{prop}
\begin{proof}
  If $Q$ is a section (quotient of a subgroup) of $G$, then it satisfies
  the same condition on soluble subgroups as $G$. This follows easily
  using the well-known fact that if $Q$ is isomorphic to $H/N$, where $H$ is
  a minimal subgroup of $G$ having $Q$ as a quotient, then $N$ must be
  nilpotent. (In fact $N$ is contained in the Frattini subgroup $\Phi(H)$,
  which is known to be nilpotent. To see this, we use another
  well-known fact, namely that $\Phi(H)$ is the intersection of the
  maximal subgroups of $H$. If such a maximal subgroup $M$ does not
  contain $N$, then $MN=H$ and hence $H/N$ is isomorphic to
  $M\big/(H\cap N)$, a contradiction.)

  We claim that if the section $Q$ is isomorphic
  to $\PSL(2,q)$, then we have $q\le 2t+1$.
  Indeed, assume the contrary.
  $Q$ is the quotient of $\SL(2,q)$
  by its center $Z=\{I,-I\}$,  where $I$ denotes the unit matrix.
  (In characteristic two, $I=-I$ and $Z=\{I\}$.)
  Hence $\SL(2,q)$ also satisfies the same condition on soluble
  subgroups.
  Let $B$ denote the subgroup of upper triangular matrices.
  $B$ is the semidirect product of the subgroups
  $$
  P = \left\{
    \begin{pmatrix}
      1&x\\
      0&1
    \end{pmatrix}
    ,\ x\in\BF_q\right\}
  \quad\text{and}\quad
  D = \left\{
    \begin{pmatrix}
      y&0\\
      0&y^{-1}
    \end{pmatrix}
    ,\ y\in\BF_q,\ y\ne0\right\},
  $$
  where $\BF_q$ denotes the base field.
  $P$ is normal, abelian, and
  the conjugate of an element
  $
  X=
  \begin{pmatrix}
    1&x\\
    0&1
  \end{pmatrix}
  $
  with
  $
  Y=
  \begin{pmatrix}
    y&0\\
    0&y^{-1}
  \end{pmatrix}
  $
  is
  $
  \begin{pmatrix}
    1&xy^2\\
    0&1
  \end{pmatrix}
  $.
  In particular, if $x\ne0$ then $X$ does not commute with $Y$ unless $y=\pm1$, 
  so the centralizer $\CC_B(X)$ is $PZ$,
  and therefore $\big|\CC_B(X)\big|\le 2q$.
  On the other hand $B$ is soluble,
  hence it has a nilpotent normal subgroup $N$ of index at most
  $t<\frac{q-1}2$. But $|B|=q(q-1)$, so $|N|>q-1=|D|$,
  hence $N$ intersects $P$ nontrivially.
  Now $P$ is an abelian Sylow subgroup of $B$, hence
  $P\cap N$ is an abelian Sylow subgroup of
  the nilpotent normal subgroup $N$,
  hence $N$ centralizes $N\cap P$.
  By the above estimate $|N|\le 2q$,
  hence $|B|\le|N|t<2q\frac{q-1}2=|B|$, a contradiction.
  
  Now let $S_i$ be one
  of the simple groups in the statement of \fref{thm:2-1} and assume
  that $S_i$ is not isomorphic to a Suzuki group. By a result of
  Liebeck, Nikolov and Shalev \cite{Liebeck_Nikolov_Shalev}, apart from the Suzuki
  groups, every finite simple group of Lie type
  over a field of $q$
  elements can be written as a product of $cl^2$ 
  subgroups isomorphic to
  $SL(2,q)$ or $PSL(2,q)$, where $l$ is the Lie rank of the group,
  and $c$ is some absolute constant.  By the
  above claim, for any such subgroup, we have $q\le 2t+1$,
  and $l \le r$.
  Hence we have
  $\big|SL(2,q)\big|\le (2t+1)^3$.  It follows that
  $|S_i|\le (2t+1)^{3cr^2}$.
  This implies that the product of the orders of the simple groups
  $S_i$ which are not Suzuki groups is at most $(2t+1)^{3cr^3}$.
  If $S_i$ is
  a Suzuki group, then it has bounded order (since the size of the
  field of definition is bounded).  Our statement follows.
\end{proof}

Next we will describe the structure of finite soluble groups of
bounded rank. For this we need the following standard description of
soluble groups with trivial Frattini subgroup
(see \cite[III-4.2 and III-4.5]{Huppert_I}).

\begin{lem}
  \label{lem:2-3}
  Let $G$ be a finite soluble group and
  suppose that $\Phi(G)=1$.
  Then the Fitting subgroup $F=F(G)$ of $G$ is a direct product of minimal
  elementary abelian normal subgroups of $G$, say
  $F=E_1\times\dots\times E_r$. The
  group $G$ acts upon $F$ by conjugation;
  the kernel of this action is $F$.
  This action defines an embedding of $G/F$ into
  $\Aut(E_1)\times\dots\times\Aut(E_r)$
  in such a way that $G/F$ induces an irreducible
  linear group $G_i=G\big/\CC_G(E_i)$ on each of the vector spaces $E_i$.
\end{lem}

We will use some results concerning the structure of primitive linear groups, which are essentially due to Suprunenko \cite{Suprunenko_Matrix_groups}, see \cite[Lemma 2.2]{Vdovin} for a concise description.
Recall that a subgroup of $G\le \GL(n,q)$ is an \emph{imprimitive linear} group if
the space $V=V(k,q)$ decomposes as a direct sum
$V_1\times\dots\times V_k$ of subspaces
$(k\ge2)$ and $G$ permutes the $V_i$'s.
It is a \emph{primitive linear} group
if no such $G$-invariant decomposition exists.

\begin{lem}[\cite{Suprunenko_Matrix_groups}, \cite{Vdovin}]
  \label{lem:2-4}
  Let $G < \GL(n,q)$
  be a maximal soluble primitive linear group.
  Then the following statements hold.
  \begin{enumerate}[\indent(a)]
  \item 
    There exists a unique maximal normal abelian subgroup $A$ of $G$.
  \item
    $A$ is isomorphic to the multiplicative group of non-zero elements of an extension
    $\BF_{q^k}$ of $\BF_{q}$, where $k$ divides $n$.
    In particular, $A$ is a cyclic
    group of order $q^k -1$.
  \item
    Setting $C=\CC_G(A)$, the quotient group $G/C$
    is isomorphic to a subgroup of the Galois group $\Gal(\BF_{q^k}\!\!:\!\BF_q)$.
    In particular, $\big|G/C\big|\le k$.
  \item
    Denoting by $B$ a maximal subgroup of $C$
    such that $B/A$ is a maximal normal abelian subgroup of $G/A$,
    the index
    $\big|B/A\big|$ equals $(n/k)^2$.
  \item
    $B/A =F(C/A)$. In particular, $C/B$ embeds into
    $\Aut(B/A)$.
  \end{enumerate}
\end{lem}

\begin{cor}
  \label{cor:2.5} 
  If $G < \GL(n,q)$ is a soluble primitive linear group,
  then it has a cyclic
  normal subgroup $A$ of order at most $q^n$
  such that $q$ and $|A|$ are relatively prime
  and the index of $A$ in $G$ is bounded in terms of $n$.
\end{cor}
\begin{proof}
  We may clearly assume that $G$ is actually a maximal primitive
  soluble group. The index of the cyclic subgroup $A$ is at most
  $k(n/k)^2\big|\Aut(B/A)\big|$.
  But the order of $\Aut(B/A)$ is bounded in terms of $n$
  since $\big|B/A\big|$ is bounded in terms of $n$.
  Our statement follows.
\end{proof}

This, in turn, implies the following.

\begin{prop}
  \label{prop:2-6}
  If $G < \GL(n,q)$ is a
  soluble completely reducible linear group,
  then it has an abelian normal subgroup $A$ of order at most $q^n$ such
  that $q$ and $|A|$ are relatively prime
  and the index of $A$ in $G$ is bounded in terms of $n$.
\end{prop}

\begin{proof}
  It is enough to show that $G$ has a (not necessarily normal) abelian
  subgroup $B$ with the required properties, since then we can take
  $A$ to be the normal core of $B$ in $G$.  Since $G$ is a subdirect
  product of irreducible linear groups of degrees say $d_1,\dots,d_m$
  with $\sum d_i=n$,
  we may clearly assume that in fact $G$ is irreducible.
  There exists a non-refinable $G$-invariant decomposition
  $V=V_1\times\dots\times V_k$
  into $l$-dimensional subspaces with $n=kl$.
  Now $G$ permutes the subspaces $V_i$ and the kernel
  $N$ of this action has index at most $n!$ in $G$.  If
  $N_i$ is the normalizer and $C_i$ is the centralizer of $V_i$ in
  $G$, then $G_i=N_i/C_i$ acts primitively on $V_i$. Hence
  $G_i$ has a cyclic normal subgroup of order at most
  $q^l$ and coprime to $q$ such that its index in
  $G_i$ is bounded in terms of $n$.  It is clear that
  $N$ is embeddable into $G_1\times\dots\times G_k$ and it follows that
  $N$ has an abelian subgroup $B$ of order at most
  $q^n$, such that $q$ and
  $|B|$ are relatively prime and the index of $B$ in
  $N$ is bounded in terms of
  $n$ (since the same holds for the direct product
  $G_1\times\dots\times G_k$ itself).
  This proves the proposition.
\end{proof}

We call a finite group of the form $E\semidirect C$,
where $E$ is an elementary abelian $p$-group
and $C$ is a cyclic group of prime power order acting
faithfully on $E$ such that $p$ does not divide $|C|$,
a group of \emph{affine cyclic type}.
We denote such an extension group by
$\Aff(E,C)$.

We will use the following
(see \cite[III-3.4 and III-4.5]{Huppert_I}).
\begin{lem}
  \label{lem:2-7}
  Let $G$ be a finite group. Then we have
  \begin{enumerate}[\indent(a)]
  \item
    $\Phi\big(G\big/\Phi(G)\big)=1$ and
  \item
    $F\big(G/\Phi(G)\big)=F(G)\big/\Phi(G)$.
  \end{enumerate}
  
\end{lem}

Next we prove the promised structure theorem for soluble groups of
bounded rank.
\begin{thm}
  \label{thm:2-8}
  Let $G$ be a soluble group of rank $r$. Then
  $G\big/F(G)$ has an abelian normal subgroup $A$
  of order at most $\big|F(G)\big/\Phi(G)\big|$
  and of index bounded in terms of $r$.
  Moreover, if $G$ has no section
  $\Aff(E,C)$ of affine cyclic type with $|C|> T$,
  then we have $|A|\le (T!)^r$.
\end{thm}

\begin{proof}
  By \fref{lem:2-7}, it is sufficient to consider soluble groups with
  trivial Frattini subgroup.  We use \fref{lem:2-3} and 
  the notation of that lemma.  By \fref{prop:2-6}, each of the quotient
  groups $G_i$ has an abelian normal subgroup $A_i$ of size at most $E_i$ and
  of index bounded in terms of $r$, say at most $f(r)$.
  Moreover, $|E_i|$ and $|A_i|$ are relatively prime.
  Hence, by our conditions,
  the orders of the elements of $A_i$
  are not divisible by prime powers larger than $T$ and 
  therefore the exponent of $A_i$ divides $T!$.
  Consider the inverse images $\hat A_i$ of the subgroups $A_i$ in $G/F(G)$.
  Let $A$ be the intersection of the subgroups $\hat A_i$.
  Since $G/F(G)$ is embeddable into the
  direct product of the $G_i$,
  and $A$ corresponds to a subgroup
  of the direct product of the $A_i$,
  we see that $A$ is an abelian normal subgroup
  of size at most $|E_1|\cdots|E_k|=\big|F(G)\big/\Phi(G)\big|$.
  Moreover, the exponent of $A$ divides $T!$.

  Since $G$ itself can be generated by $r$ elements, the intersection of all
  of its subgroups of index at most $f(r)$ has bounded index in terms of $r$
  (see \cite[Cor. 1.1.2]{Lubotzky_Segal}),
  hence the same is true for $G/A$.

  Now $A$ is  abelian of rank at most $r$,
  hence it is the product of at most $r$ cyclic subgroups
  of order at most $T!$,
  and therefore $|A|\le(T!)^r$.
  The proof is complete.
\end{proof}

Combining the previous facts we prove our first result which
characterizes nilpotent by bounded groups among groups of bounded
rank.

\begin{thm}
  \label{thm:2-9}
  Let $G$ be a finite group of rank $r$.
  Assume that $G$ has no section $\Aff(E,C)$ of affine cyclic type
  with $|C|> T$. Then $G$ has a
  nilpotent normal subgroup of index bounded in terms of $r$ and $T$.
\end{thm}
\begin{proof}
  By \fref{thm:2-8}, if $H$ is a soluble group satisfying the above
  conditions, then
  $\big|H\big/F(H)\big|$ is bounded in terms of $r$ and $T$.
  In particular,
  this holds for all soluble subgroups of a group $G$ satisfying these
  conditions. Using \fref{prop:2-2} we obtain our statement.
\end{proof}

Note that if a finite group has a nilpotent normal subgroup of index $T$,
then it can not have a section $\Aff(E,C)$ of affine cyclic type with
$|C|> T$.

We will also give another characterization in terms of certain
excluded subgroups.

We need the following.

\begin{lem}
  \label{lem:2-10} 
  Let $A\isom\Aff(E,C)$ be a group of affine cyclic type.
  Assume that for a
  finite group $G$, we have $G/N\isom A$
  and that no proper subgroup of $G$ has a
  quotient isomorphic to $A$.
  Then $G$ is a group of the form $P\semidirect \BZ_m$,
  where $P$ is a $p$-group, $m$ is a power of a prime $q\ne p$,
  and the image of $\BZ_m$ in $\Aut(P)$ has order at least $|C|$.
\end{lem}
\begin{proof}

  By our condition on the minimality of $G$, the normal subgroup $N$ is
  contained in the Frattini subgroup $\Phi(G)$ of $G$ (see the proof of
  \fref{prop:2-2}), hence it is nilpotent. Moreover, the pre-image $\hat E$ of $E$ in $G$ is a normal subgroup of $G$ which satisfies $\hat E'\subseteq N\subseteq \Phi(G)$. Hence $\hat E$ is nilpotent by Wielandt's theorem (\cite[III-3.11]{Huppert_I}).

  If $S$ is a Sylow subgroup
  of $\hat E$ of order coprime to $|E|$ and $|C|$, then it is normal in $G$
  (since it is a characteristic subgroup of $\hat E$) and by the
  Schur-Zassenhaus theorem
  (see \cite[Ch. 2, (8.10)]{Suzuki_I})
   it has a complement. But then this complement has a quotient isomorphic to $A$,
  a contradiction.

  The Sylow $p$-subgroup $S_p$ is also normal in $G$, and it follows that the quotient $G/S_p$ is a $q$-group for some prime $q\neq p$.
  This quotient has a cyclic subgroup which projects onto $C$.
  By our minimality condition, 
  $G$ is actually the split extension of $S_p$
  by this cyclic subgroup. Our statement follows.
\end{proof}

Recall that a $p$-group $P$ is said to be a special $p$-group if either
$P$ is an elementary abelian $p$-group or we have $\Phi(G)=\CZ(G)=G'$
and $G'$ is elementary abelian. In particular, such a group has class
at most $2$ and exponent at most $p^2$.

We need a well-known result of Hall-Higman
(see \cite[Ch. 4, (4.19)]{Suzuki_II}).

\begin{thm}[Hall-Higman]
  \label{thm:2-11}
  Assume that the group $A$ acts on a group $H$, where $A$ and $H$ have
  coprime order.  Suppose that a subgroup $R$ of $A$ acts on $H$
  non-trivially. Let $P$ be an $A$-invariant subgroup of $H$ that is
  minimal among the $A$-invariant subgroups of $H$ on which
  $R$ acts nontrivially. Then $P$ is a special
  $p$-group for some prime $p$.  The group $R$ acts on
  $\Phi(P)$ trivially, and the group $A$ acts on
  $P\big/\Phi(P)$ irreducibly.
\end{thm}

\begin{rem}
  \label{rem:2-12}
  This theorem implies immediately that in the conclusion of
  \fref{lem:2-10} we may take $P$ to be a special
  $p$-group (we apply \fref{thm:2-11} by setting $A$ to be $\BZ_m$,
  $H$ to be $P$ and $R$ to be the smallest subgroup of
  $\BZ_m$ which acts on $H$ nontrivially).
\end{rem}

We arrived at the central result of this section.

\begin{thm} [same as \fref{thm:Nilpotent-or-Special-by-cyclic_v0}]
  \label{thm:Nilpotent-or-Special-by-cyclic}
  Let $G$ be a finite group of rank
  $r$. Then for all $T > 0$, there is an integer $I(r,T) > 0$ such that
  one of the following holds.
  \begin{enumerate}[(a)]
  \item
    $G$ has a nilpotent normal subgroup of index at most $I(r,T)$.
  \item \label{item:3332}
    For some distinct primes $p$ and $q$ the group $G$ has a
    subgroup of the form $P\semidirect \BZ_m$,
    where $P$ is a special $p$-group, $m$ is a
    power of $q$, and the image of $\BZ_m$ in $\Aut(P)$ has order at
    least $T$.
    In particular, we have $|P|\le p^{2r}$.
\end{enumerate}
\end{thm}
\begin{proof}
  If $G$ has a section of affine cyclic type with $|C|> T$, then \fref{item:3332} holds
  by \fref{lem:2-10} and \fref{rem:2-12}.
  Otherwise, by \fref{thm:2-9}, $G$ has a nilpotent normal subgroup
  of index bounded in terms of $r$ and $T$.
\end{proof}

As a corollary of \fref{thm:Nilpotent-or-Special-by-cyclic} we will also show that if a finite
group $G$ of bounded rank does not contain certain abelian-by-cyclic
subgroups, then $G$ is abelian-by-bounded. While we do not use this
corollary in our paper, it should be useful for proving analogues of
Jordan's Theorem in various situations.  The corollary was explicitly
stated in a talk by the second named author  \cite{Pyber-lecture}.
A somewhat weaker result has independently (see the
last sentence on page 824 of \cite{Riera_Turull} concerning this issue) been
obtained by Mundet i Riera and Turull as the main theorem of \cite{Riera_Turull}. Our
arguments are very much different from those in \cite{Riera_Turull}.

We need to prove two lemmas on finite $p$-groups.

The proof of the first lemma is based on the following useful result
of Chermak-Delgado (see \cite[1.41]{Isaacs}).

\begin{prop}
  \label{prop:2-14}
  Let $G$ be a finite group.
  Then $G$ has a characteristic abelian subgroup $N$ such that
  $|G:N|\le|G:A|^2$ for every abelian subgroup $A$ of $G$.
\end{prop}

\begin{lem}
  \label{lem:2-15}
  Let $P$ be a finite $p$-group. Assume
  that all metabelian normal subgroups of $P$ have an abelian subgroup
  of index at most $t$. Then $P$ has an abelian normal subgroup of
  index at most $t^{4\log t}$.
\end{lem}
\begin{proof}
  Let $A$ be an abelian normal subgroup of maximal order in $P$. Let
  $\tilde B$ be an abelian normal subgroup of maximal order in
  $\tilde P=P/A$ and set $b=|\tilde B|$. By a classical result of
  Burnside-Miller
  (see \cite[Cor. 2 of  Ch. 2, Thm. 1.17]{Suzuki_I}),
  we have
  $|\tilde P|\le p^{(\log_p b)(\log_pb+1)/2}\leq b^{\log b}$.  The inverse image $B$ of $\tilde B$ in
  $P$ is a metabelian normal subgroup of $P$,
  hence by assumption
  it has an abelian subgroup of index at most $t$.
  By \fref{prop:2-14}, \
  $B$ has an abelian characteristic subgroup $C$ of index at most
  $t^2$.  But $C$ (as a characteristic subgroup of a normal subgroup)
  is a normal subgroup of $P$. Hence by assumption we have $|C|\le|A|$,
  and therefore $t^2\ge b$.  We obtain
  $|\tilde P|\le(t^2)^{2\log t} = t^{4\log t}$, as required.
\end{proof}

\begin{lem}
  \label{lem:2-16}
   Let $P$ be a finite metabelian $p$-group of rank $r$. Assume that
   all abelian-by-cyclic subgroups of $P$ have an abelian
   subgroup of index at most $t$. Then $P$ has an abelian normal
   subgroup of index at most $t^{2r}$.
\end{lem}
\begin{proof}
  By a result of Gillam \cite{Gillam}, a metabelian group $P$ has an
  abelian normal subgroup $A$ such that the order of any abelian
  subgroup of $P$ is at most $|A|$.
 It is easy to see that a metabelian group of rank $r$ is a product of $2r$ cyclic subgroups. Hence the quotient
  $\tilde P=P/A$ is a product of $2r$ cyclic groups
  $\tilde C_1,\tilde C_2,...,\tilde C_{2r}$. The inverse image $C_i$ of
  $\tilde C_i$ in $P$ is an abelian-by-cyclic group, and by
  assumption, 
  $A$ is an abelian subgroup of $C_i$ of largest order.  By our
  condition, we have
  $|\tilde C_i|\le\big|{C_i:A}\big|\le t$ (for $i=1,...,2r$).
  Hence $|{P:A}|\le t^{2r}$, as required.
\end{proof}

\begin{cor}
  \label{cor:2.17}
  Let $G$ be a finite group of rank $r$. Then for all $T > 0$, there is
  an integer $J(r,T) > 0$ such that one of the following holds:
  \begin{enumerate}[\indent(a)]
  \item \label{item:3}
    $G$ has an abelian normal subgroup of index at most $J(r,T)$.
  \item \label{item:4}
    For some distinct primes $p$ and $q$ the group $G$ contains
    the split extension of an elementary abelian $p$-group $E$ by a
    cyclic group $C$ of order $q^t$, such that the image of $C$ in
    $\Aut(E)$ has order at least $T$.
  \item \label{item:5}
    $G$ contains an abelian-by-cyclic $p$-subgroup $P$ which does
    not have an abelian normal subgroup of index at most $T$.
  \end{enumerate}
\end{cor}
\begin{proof}
  Assume that \fref{item:5} does not hold. We observe that for
  $p > T$ this implies that any $p$-subgroup of $G$ is abelian.
  Assume also that \fref{item:4} does not hold.

  We first show using
  \fref{thm:Nilpotent-or-Special-by-cyclic} that in this case $G$ has a nilpotent normal
  subgroup of index bounded in terms of $r$ and $T$.  Let $H$ be a
  subgroup of $G$ which is an extension of a special $p$-group $P$ by
  $\BZ_m$ ($m$ a power of some prime $q$ distinct from $p$). As noted
  above, either $P$ is elementary abelian or it has order at most
  $T^{2r}$. If a cyclic group $\BZ_m$ acts on $P$ then in the first case
  the image of $\BZ_m$ in $\Aut(P)$ has order at most $T$ by our
  assumption (that \fref{item:4} does not hold). In the second case
  the order of the image of $\BZ_m$ in $\Aut(P)$ is at most
  $|P|\le T^{2r}$.
  Now \fref{thm:Nilpotent-or-Special-by-cyclic}
  implies that $G$ has a nilpotent normal subgroup of index at most
  $I\big(r,T^{2r}\big)$.
  In particular this is a bound for the index of $F(G)$ in $G$.

  Now let $G$ be a $p$-group for some prime $p\le T$ for
  which \fref{item:5} does not hold, i.e., assume that all
  abelian-by-cyclic subgroups of $G$ contain an abelian normal
  subgroup of index at most $T$.  If $H$ is any metabelian subgroup of
  $G$, then by \fref{lem:2-16}, $H$ has an abelian normal subgroup of
  index at most $T^{2r}$. Using \fref{lem:2-15} (with $t=T^{2r}$), we see
  that $G$ itself has an abelian normal subgroup of index at most
  $(T^{2r})^{4\log T^{2r}}=T^{16 r^2\log T}$.

  Consider now the general case.
  The Fitting subgroup $F(G)$ is the direct product of its Sylow
  $p$-subgroups. For $p> T$ these are abelian and for $p\le T$ we have
  just proved that they have an abelian normal subgroup of index
  bounded in terms of $r$ and $T$. Hence the same holds for $F(G)$
  itself. But we have already proved that the index of $F(G)$ is
  bounded in terms of $r$ and $T$.  The proof of the corollary is
  complete.
\end{proof}

\section{Preliminaries in topology}
\label{sec:preliminaries}

In this section, we collect a number of basic results from
topology. Though the main theorems in the introduction give bounds on
the complexity of finite group actions on a topological manifold in
terms of the dimension and the integer homology group of the manifold,
we also want to use homology with coefficients in the $p$-element field $\BF_p$,
and cohomology.

The singular homology of a topological space $X$ with integer coefficients is
denoted by $H_*(X;\BZ)$, homology of $X$ with coefficients in the
field $\BF_p$ is denoted by $H_*(X;\BF_p)$.
For a sheaf $\CA$ on $X$, we denote by $H_c^*(X;\CA)$ the sheaf
cohomology of $\CA$ with compact support in the sense of Grothendieck
\cite{grothendieck1957abelianCategories}, see also
\cite[Chapter~II]{bredon2012sheaf}.

For a finite group $G$, the group cohomology of a $G$-module $M$
is denoted by $H^*(G;M)$, see \cite[Section~IV-2.]{borel1960seminar}.

Let $M$ be a $d$-dimensional topological manifold.  For every prime $p$,   
$M$ has a mod~\!$p$ \emph{orientation sheaf}  $\Orientation[M]$, or shortly $\Orientation$,
which is the sheafification of the presheaf which assigns to an open
subset $U\subseteq M$ the relative singular homology group
$H_d(M,M\setminus U;\BF_p)$,
see \cite[Example I-1.11]{bredon2012sheaf}.

Although our main results deal with topological manifolds,
in \fref{sec:Reduction-to-free-action} we need to work with
more general spaces, cohomology manifolds over the field $\BF_p$.
Cohomology manifolds are defined e.g. in
\cite[Definition V-16.7]{bredon2012sheaf},
they are homology manifolds 
by \cite[Definition V-16.8]{bredon2012sheaf}.
The orientation sheaf of a homology manifold over $\BF_p$ is
the sheafification of the presheaf which assigns to an open subset $U$
the Borel-Moore homology of $U$,
see \cite[Definition V-9.1]{bredon2012sheaf}
and \cite[page 293]{bredon2012sheaf}.
Topological manifolds are cohomology manifolds over any field,
and the above definitions of the orientation sheaf are compatible.

We need some general notation for group actions.
\begin{defn}
  Let $G$ be a group acting on a set $X$.
  Denote by $G_x$ the  stabilizer subgroup
  of an element $x\in X$.
  For subgroups $H\le G$, we denote by $X^H$ the
  fixed point subset of $H$.
  We will study the set of all stabilizer subgroups
  $$
  \Stab(G,X) = \left\{G_x\,\big|\,x\in X\right\}.
  $$
\end{defn}

\begin{defn}
  Let $G$ be a group. A \emph{$G$-manifold} is a topological manifold
  equipped with a continuous  action of $G$. The $G$-manifold is \emph{effective}
if the $G$-action is effective.
\end{defn}

Many of our constructions involving the orientation sheaf will be
\emph{functorial} with respect to open embeddings due to the following
simple proposition, which follows from the fact that the definition
of the orientation sheaf is local, and does not involve any choices.

\begin{prop} \label{prop:Orientation-canonical}
  For any open embedding $f\colon N\to M$ of cohomology manifolds over $\BF_p$,
  there is a natural isomorphism
  $$
  \Orientation[N]\stackrel\isom\longrightarrow
  f^*\Orientation[M].
  $$
\end{prop}

\begin{prop}\label{prop:G-structure on orientaion_sheaf}
  If a finite group $G$ acts continuously on a topological manifold $M$,
  then the action can be lifted canonically to a $G$-action
  on the sheaf $\Orientation[M]$, and on its cohomology $H_c^*(M;\Orientation[M])$.
  Moreover, if the $G$-action on $M$ is free and $f\colon M\to M/G$ denotes
  the quotient map, then $\Orientation[M]\isom f^*\Orientation[M/G]$ canonically.
\end{prop}
\begin{proof}
	Follows immediately from \fref{prop:Orientation-canonical}.
\end{proof}

\begin{prop}  [Poincar\'e duality]
  \label{prop:generalised-Poincare-duality}
  Let $M$ be a $d$-dimensional topological manifold.
  For all $0\le i\le d$, there are canonical isomorphisms
  $$
  \Delta_M\colon H_c^i(M;\Orientation)\stackrel\isom\longrightarrow
  H_{d-i}(M;\BF_p).
  $$
  Moreover, every open embedding
  $f\colon N\hookrightarrow M$ 
 induces a commutative diagram

\begin{equation}
  \label{eq:1}
  \xymatrix@C40pt{
    H_c^i(N;\Orientation[N]) \ar[r] \ar[d]^{\Delta_N}
    &H_c^i(M;\Orientation[M]) \ar[d]^{\Delta_M}
    \\
    H_{d-i}(N;\BF_p) \ar[r]
    &H_{d-i}(M;\BF_p)\quad .
  }
\end{equation}

\end{prop}
\begin{proof}
  Let $\scrA$ be the constant sheaf on $M$  with $\BF_p$ coefficients,
  then $\Orientation[M]\otimes\scrA\isom\Orientation[M]$,
  and our $\Delta_M$ is constructed in
  \cite[Theorem V-9.2]{bredon2012sheaf}.
  Moreover, let $\scrA'\le\scrA$ be the subsheaf which is equal to
  $\scrA$ on the image of $f$ and $0$ outside.
  \fref{prop:Orientation-canonical} implies that
  $H_c^i(N;\Orientation[N])\isom H_c^i(M;\Orientation[M]\otimes\scrA')$,
  hence the diagram \fref{eq:1} is also constructed in
  \\ \cite[Theorem~V-9.2]{bredon2012sheaf}.
  We remark that Borel-Moore
  homology with compact supports, which appears in the cited theorem
  of \cite{bredon2012sheaf},  coincides with singular homology in the
  case of topological manifolds.
\end{proof}

\begin{prop} \label{prop:Hd-cohomology-manifold}
  Let $M$ be a $d$-dimensional cohomology manifold.
  Then $H_c^d(M;\Orientation)$ is isomorphic to the free
  $\BF_p$-module on the connected components of $M$,
  and this isomorphism is a natural transformation for open embeddings.
  In particular,
  $\dim H_c^d(M;\Orientation)$ is the cardinality of the 
  connected components of $M$.
\end{prop}
\begin{proof}
By Poincar\'e duality \cite[V-9.2]{bredon2012sheaf}, $H_c^d(M,
\Orientation)$ is isomorphic to the Borel-Moore homology $H_0^c(M,
\BF_p)$ with compact support. Then our statement is a special case of
\cite[V-5.14]{bredon2012sheaf}.
\end{proof}

Smith theory describes the fixed point set of finite $p$-groups
acting continuously on topological manifolds.
In particular, these are disjoint unions of cohomology manifolds.

\begin{prop} \label{prop:fixed-point-set-is-manifold}
  Let $G$ be a finite $p$-group
  acting continuously and effectively on a topological manifold $M$.
  Then each connected component $F$ of $M^G$ is
  a cohomology manifold over $\BF_p$,
  and we have an isomorphism
  \[
    \Orientation[M]\big|_F\isom\Orientation[F]
  \]
  which is a natural transformation with respect to
  $G$-equivariant open embeddings.
  If $M$ is connected and $|G|>2$ then
  $M\setminus M^G$ is connected.
  
  Furthermore,  we have the inequalities
  \begin{eqnarray*}
    \dim H_c^*\big(\Fix{M}{G};\Orientation\big)
    &\le&
      \dim H_c^*(M;\Orientation),
    \\
    \dim H_c^*\big(M\setminus\Fix{M}{G};\Orientation\big)
    &\le&
      2\,\dim H_c^*(M;\Orientation).
  \end{eqnarray*}
\end{prop}
\begin{proof}
  Most of this is contained in \cite[Proposition~2.5]{CMPS}.
  If $M$ is connected and $|G|>2$ then $\dim_\BZ M\le \dim(M)-2$
  by  \cite[Theorem~V/2.6]{borel1960seminar} and
  $M\setminus M^G$ is connected by \cite[Corollary~I/4.7]{borel1960seminar}.
\end{proof}

For an elementary abelian $p$-group acting on a manifold,
Borel's Fixed Point Formula \cite[Theorem XIII-4.3]{borel1960seminar}
helps us to understand the local behavior of stabilizer subgroups and
their fixed point submanifolds.

\begin{prop} [Borel]
	\label{prop:Borel-fixpoint-formula}
	Let $M$ be  a topological $d$-manifold, and
	let $G$ be an elementary abelian $p$-group
	acting continuously and effectively on $M$.
	Let $x\in \Fix{M}{G}$ be a fixed point of $G$.
	For a subgroup $H\le G$, denote by  $d(H)$ the dimension
	of the connected component of $x$ in $M^H$.
	Then
	$$
	d-d(G) = {\sum}_H \big( d(H) - d(G) \big),
	$$
	where $H$ runs through the subgroups of $G$ of index $p$.
\end{prop}

For a $p$-group acting on a manifold,
Ignasi Mundet i Riera and the authors investigated
the global behavior of stabilizer subgroups and
their fixed point submanifolds.
Their main result is the following.

\begin{prop} [{\cite[Theorem~5.1]{CMPS}}]
  \label{prop:number-of-stapilizers-is-bounded}
  For all integers $d,B\ge0$, there is a bound $\tilde C=\tilde C(d,B)$
  with the following property.\\
  Let $M$ be a topological $d$-manifold such that
  $\dim H_c^*(M;\Orientation)\le B$.
  Then each finite $p$-group $G$ acting continuously on $M$
  has a characteristic subgroup $H\le G$ of index at most $\tilde C$
  such that
  $$
  \big|\Stab(H,M)\big|\le\tilde C.
  $$  
\end{prop}

An old result of Mann-Su
\cite{Mann_Su} gives an upper bound on the rank
of an elementary abelian $p$-group $G$ acting on a compact
topological manifold $M$.
Ignasi Mundet i Riera and the authors extended this
to not necessarily compact manifolds and to arbitrary finite groups $G$.
\begin{thm}[{\cite[Theorem~1.8]{CMPS}}]\label{thm:Mann-Su_v0} 
  Let $G$ be  a finite group  acting continuously
  and effectively on a topological manifold $M$ such that
  $H_*(M;\BZ)$ is finitely generated.
  Then the rank of $G$ 
  is bounded in terms of $\dim(M)$ and $
  H_*(M;\BZ)$.
\end{thm}

\subsection{Equivariant cohomology}
The most important topological tools in our paper, equivariant cohomology and a spectral sequence converging to it, were essentially introduced by Borel (see \cite[section~IV-3.1]{borel1960seminar}). For discrete group actions, a different but equivalent approach to these notions is due to Grothendieck \cite{grothendieck1957abelianCategories}. We prefer to use the latter approach here.

\begin{defn}
	If $G$ is a discrete group, $X$ is a topological space equipped with a continuous action of $G$, or shortly a \emph{$G$-space}, and $R$ is a commutative unital ring, then a sheaf of $R$-modules $\pi\colon \mathcal A\to X$  on $X$ is a \emph{$G$-sheaf} if we are given a continuous action $\tau$ of $G$ on the sheaf space $\mathcal A$ for which $\tau_g(\mathcal A_x)={\mathcal A_{gx}}$ and the restriction of $\tau_g|_{\mathcal A_x}\colon {\mathcal A_x} \to {\mathcal A_{gx}} $ is a module isomorphism for every $g\in G$ and $x\in X$. 
\end{defn}

The most important examples of $G$-sheaves in this paper are the
orientation sheaves of topological manifolds with a group action
(see \fref{prop:G-structure on orientaion_sheaf}).

Suppose now that $G$ is a discrete group.
If we are given a $G$-sheaf $\mathcal A$ of $R$-modules on the
$G$-space $X$,
then we can consider the $R$-module $\Gamma_{c}^G(\mathcal A)$ of
$G$-equivariant sections of $\mathcal A$ with compact support.
The functor $\Gamma_{c}^G(-)$ is left exact, so it gives rise to the
right derived functors $\rmR^i\Gamma_{c}^G(-)$.

\begin{defn}[Grothendieck {\cite[Sec.~5.7]{grothendieck1957abelianCategories}}]
	The $i$-dimensional equivariant sheaf cohomology of the
	$G$-sheaf $\mathcal A$ of $R$-modules on the $G$-space $X$ with
	compact support is the $R$-module\footnote{ Grothendieck used the
		notation $H_{c}^i(X;G,\mathcal A)$ for this group.}
	\[
	H_{G,c}^i(X;\mathcal A)=\rmR^i\Gamma_{c}^G(\mathcal A).
	\]
\end{defn}

Group cohomology is an important special case of equivariant cohomology.   
\begin{defn}\label{defn:Group-cohomology}
	Let $G$ be a discrete group, $V$ an $RG$-module.
	Let  $\mathcal V$ be the $G$-sheaf on a one point space $\{p\}$,
	with stalk $V$ at $p$, where $G$ acts on the stalk via the
	$RG$-module structure.
	Then the \emph{group cohomology of  $G$ with coefficients in  $V$} is
	$$
	H^*(G;V)=H_{G,c}^*\big(\{p\};\mathcal V\big).
	$$	
\end{defn}

In the following cases the equivariant cohomology can be computed easily.

\begin{prop}\label{prop:free-action-cohomology}
  Let $M$ be a cohomology manifold over $\BF_p$
  and let $G$ be a finite group
  acting continuously on $M$.
  \begin{enumerate}[\indent(a)]
  \item \label{item:8}
    If the action is free, then
    $$
    H_{G,c}^*\big(M;\Orientation\big) \isom
    H_c^*\big(M/G;\Orientation\big).
    $$
  \item \label{item:9}
    If the action is trivial then
    $$
    H_{G,c}^*\big(M;\Orientation\big) \isom
    H_c^*\big(M;\Orientation\big)\otimes H^*(G;\BF_p).
    $$
  \end{enumerate}
  Both isomorphisms are natural transformations
  for $G$-equivariant open embeddings.
\end{prop}
\begin{proof}
  For \fref{item:8} we refer to
  \cite[IV~(32)]{bredon2012sheaf} or \cite[Prop.~2.7~(a)]{CMPS}.
  \fref{item:9} is a special case of the K\"unneth formula,
  we refer to \cite[Proposition~2.7]{CMPS} for a complete proof.
\end{proof}

In the general case, the following spectral sequence is a basic tool to compute equivariant cohomology.
\begin{prop}[{\cite[Sec.~5.7]{grothendieck1957abelianCategories}}]
	\label{prop:Leray-spectral-sequence}
	Let $G$ be a finite group acting continuously on a topological
        manifold $M$.
	Then there exists a spectral sequence
	\begin{equation*}
	\BE_t^{i,j} \Longrightarrow H_{G,c}^{i+j}(M;\Orientation),
	\quad\quad
	\BE_2^{i,j}=H^i\big(G;H_c^j(M;\Orientation)\big),  
	\end{equation*}
	which is functorial in $M$ with respect to $G$-equivariant open embeddings.
\end{prop}

The following long exact sequence is another useful tool for calculations.

\begin{prop} \label{prop:long-exact-sequence-of-closed-subset}
  Let $G$ be a finite $p$-group acting continuously on a topological
  manifold $M$.
  With the notation $F=\Fix{M}{G}$
  and $U=M\setminus F$, 
  we have a long exact sequence
  $$
  \cdots
  H_{G,c}^n\big(U;\Orientation[U]\big)\to
  H_{G,c}^n\big(M;\Orientation[M]\big)\to
  H_{G,c}^n\big(F;\Orientation[F]\big)\to
  H_{G,c}^{n+1}\big(U;\Orientation[U]\big)\cdots
  $$
  which is functorial for $G$-equivariant open embeddings.
\end{prop}
\begin{proof}
  This is a special case of the long exact sequence of cohomology with
  compact support. See \cite[Proposition~2.6]{CMPS} for a complete proof.
\end{proof}

We recall some useful facts about the cohomology groups of elementary
abelian $p$-groups, which will be used to compute the second page of
the spectral sequence in \fref{prop:Leray-spectral-sequence}. In this
paper we do not use the ring structure of $H^*(G;\BF_p)$,
but in the proposition below, it is easier to describe $H^*(G;\BF_p)$ as a graded ring.

\begin{prop} \label{prop:cohomology-of-elementary-abelian-p-groups}
	Let $G$ be an elementary abelian $p$-group of rank $r$.
	\begin{enumerate}[\indent(a)]
        \item \label{item:6}
          $H^1(G;\BF_p)$ is naturally isomorphic to $\Hom(G,\BF_p)$.
        \item \label{item:10}
          For $p=2$, \ $H^*\big(G;\BF_2\big)$
          is the symmetric algebra $S^*\big(H^1(G;\BF_p)\big)$.
		\item \label{item:11}
		If  $p\ge3$, then
		the Bockstein homomorphism
		$\beta^1\colon H^1(G;\BF_p)\hookrightarrow H^2(G;\BF_p)$ is injective,
		and
		$$
		H^*(G;\BF_p) =
		\Lambda^*\big(H^1(G;\BF_p)\big)\otimes
		S^*\big(\beta^1(H^1(G;\BF_p))\big),
		$$
		where, for a vector space $V$ over $\BF_p$,
		$\Lambda^*(V)$ and $S^*(V)$ denote the Grassmannian algebra and the
		symmetric algebra of $V$.
		In particular,
		for all $d$, we have a natural isomorphism
		$$
		H^d(G;\BF_p) =
		\bigoplus_{l+2s=d}
		\Lambda^l\big(\Hom(G;\BF_p)\big) \otimes
		S^s\big(\Hom(G;\BF_p)\big).
		$$
		\item \label{item:12} For all $p$, we have
		$$
		\dim H^d(G;\BF_p) = \binom{d+r-1}{d} = \binom{d+r-1}{r-1}.
		$$
		\item \label{item:14}
		If $V$ is a trivial finite dimensional $\BF_pG$-module, then
		we have
		$$
		H^d(G;V) \isom H^d(G;\BF_p)\otimes V.
		$$
	\end{enumerate}
\end{prop} 
\begin{proof}
	The cohomology ring of $\BZ_p$ is described in
	\cite[IV-2.1(3)]{borel1960seminar}. Statements (a), (b), (c), (d) for $G=\BZ_p^r$ and isomorphism \fref{item:14}
	follow from the K\"unneth formula. 
\end{proof}

\section{Free actions}
\label{sec:free-actions}

In this section we study groups of the form $G=P\semidirect H$
acting effectively on a topological manifold $M$,
where $P$ is a $p$-group acting freely on $M$, and $p$ does not divide $|H|$.
First we consider the base case $P\isom\BZ_p^r$ which will serve as the starting point for our inductive arguments later on.

Our proof needs the extra technical condition that the action of $P$ on $H_c^*(M;\Orientation)$ should be trivial. However, when the manifold $M$ will be modified during the proof, we shall loose  control over the $P$-action on the cohomology of the new manifold.
The following lemma allows us to go around this problem using the $H$-action instead, which will be easier to control.

\begin{lem} \label{lem:semidirect-product-trivial-H-action}
  Let $G$ be a finite group of the form $G=A\semidirect_\alpha H$,
  where 
  $A\isom\BZ_p^r$, $\alpha \colon H\to\Aut(A)$ is an $H$-action,
  and $|H|$ is not divisible by $p$.
  Let $A'\le A$ denote the subgroup of $H$-invariant elements. Then 
  $A'$ has an $H$-invariant complement $A''\le A$.
  Whenever $G$ acts on a set $X$ so that the subgroup $H$
  acts on $X$ trivially,
  the action of  $A''$  on $X$ is trivial as well.
\end{lem}
\begin{proof}
  Since $|H|$ is not divisible by $p$, $A$ decomposes into a direct sum of irreducible $H$-modules by Maschke's theorem. The sum of the non-trivial summands in this decomposition is a good choice for $A''$.

 We have to prove that the kernel $N\normal G$ of the $G$-action on $X$ contains $A''$. By assumption $H\le N$.
  
  Let $0\neq B\le A''$ be any irreducible $H$-submodule of $A''$.
  Then $B\cap N$ is normal in $G$, so it is an $H$-submodule of $B$.
  If it were $0$, then $B$ would commute with $N$,
  hence commute with $H$, contrary to the definition of $A''$.

  Therefore $B\cap N\ne0$, hence $B\le N$ by the irreducibility of $B$.
  This holds for all irreducible submodules $B\le A''$,
  hence $A''\le N$.
\end{proof}

\begin{lem} \label{lem:elementary-p-by-cyclic-free-action}
	For all integers $r,d\ge0$, there is an integer $n_1(r,d)$
	with the following property.\\
	Let $G$ be a finite group of the form $G=A\semidirect_{\alpha} H$,
	where 
	$A\isom\BZ_p^r$, $\alpha\colon H\to\Aut(A)$ is an $H$-action,
	and $|H|$ is not divisible by $p$.
	Let $M$ be a $d$-dimensional topological manifold with an effective $G$-action.
	Suppose that
	the subgroup $A$ acts freely on $M$,
	and the induced action of the subgroup $H$
	on $H_c^*(M;\Orientation)$
	is trivial.
	Then $H$ has a subgroup $\tilde H$ of index at most $n_1(r,d)$
	commuting with $A$.
\end{lem}
\begin{proof}
  As in \fref{lem:semidirect-product-trivial-H-action},
  let $A'\le A$ denote the subgroup of $H$-invariant elements
  and $A''$ denote its $H$-invariant complement.  
  It is enough to find a subgroup $\tilde H\le H$ of bounded index
  which commute with $A''$,
  since this $\tilde H$ will commute with $A=A'\oplus A''$ as well.
  
  The $H$-action on $A''$ is a homomorphism
  $\alpha''\colon H\to\Aut(A'')$.
  The subgroup $\tilde H=\ker(\alpha'')$ commutes with
  $A''$, so we need to bound its index, that is the order of the group
  $K=\im(\alpha'')\cong H/\tilde H$.
  We shall find a bound on $\big|K\big|$ using the fact that if
  $\chi_1,\dots,\chi_c$ are representatives of the isomorphism classes
  of the irreducible representations of $K$ over the algebraic closure
  $\overline{\BF}_p$	of $\BF_p$, and $\chi_i$ acts on a vector
  space of dimension $d_i$, then
  \[\big|K\big|=\sum_{i=1}^c d_i^2\leq \Big(\sum_{i=1}^c d_i\Big)^2.\] 
	
  As $p$ does not divide $|K|$, any finite dimensional representation $\tilde \chi$ of $K$ over $\overline{\BF}_p$ can be decomposed into the direct sum of irreducible representations. We shall say that the irreducible representation \emph{$\chi_i$ is contained in $\tilde\chi$}, if the multiplicity $m_i$ of  $\chi_i$ in the decomposition $\tilde \chi\cong\bigoplus_{i=1}^c m_i\chi_i$ is positive. 
	
  By \fref{lem:semidirect-product-trivial-H-action},
  the natural $A''$-action on $H_c(M;\Orientation)$ is trivial.
  Applying \fref{prop:Leray-spectral-sequence}
  to the free $A''$-action on $M$,
  we obtain a spectral sequence
  $$
  \BE_t^{i,j} \Rightarrow H_c^{i+j}(M/A'';\Orientation),
  $$
  $$
  \BE_2^{i,j} = H^i\big(A'';H_c^j(M;\Orientation)\big) \isom
  H^i(A'';\BF_p)\otimes_{\BF_p} H_c^j(M;\Orientation).
  $$
  The group $H$ acts canonically on the entire spectral sequence,
  and the subgroup $\tilde H$ acts trivially on $\BE_2^{i,j}$. Therefore
  the $\tilde H$-action is trivial on each $\BE_t^{i,j}$,
  so the $H$ action induces a $K$-action on the entire spectral sequence.
  Denote by $\chi_t^{i,j}$ the representation of $K$ induced on $\overline{\BF}_p\otimes \BE_t^{i,j}$ and by $R^{i,j}_t$ the set of irreducible representations contained in $\chi_t^{i,j}$ . We collect some facts on the sets  $R^{i,j}_t$.
	
	(1)  The $K$-module $\BE_{t+1}^{i,j}$ is a factor of a submodule of  $\BE_{t}^{i,j}$, hence $R^{i,j}_2\supseteq R^{i,j}_3\supseteq R^{i,j}_4\supseteq\dots$  is a weakly decreasing sequence.
	
	(2) The inclusion map  $\chi\colon K\hookrightarrow\Aut(A'')\hookrightarrow \GL(r,\overline{\BF}_p)$ is a  faithful linear representation of $K$ over $\overline{\BF}_p$.	 Denote by $$\chi^*\colon K \to \GL(\Hom(\overline{\BF}_p \otimes A'',\overline{\BF}_p))$$  the dual representation of $\chi$, and let  $R^i$ be the set  of irreducible representations contained in $S^i(\chi^*)$  if $p=2$ and in  $\bigoplus_{k+2l=i}\Lambda^k(\chi^*)\otimes S^l(\chi^*)$ if $p$ is odd.   \fref{prop:cohomology-of-elementary-abelian-p-groups}
	implies that if $H_c^j(M;\Orientation)\neq 0$, then $R_2^{i,j}=R^i$, otherwise $R_2^{i,j}=\emptyset$.
	
	(3) Introduce the sets $\tilde R^i=  \bigcup_{j=0}^{i}R^{j}$. We claim that all the irreducible representations $\chi_1,\dots,\chi_c$ are contained in  the representation $\bigoplus_{l=0} ^{|K|-1} S^l(\chi^*)$, in particular, they are all contained in  $\tilde R^{2(|K|-1)}$. Indeed, as $\chi$ is a faithful representation of $K$, the fixed point set of the action of an element $g\in K\setminus\{e\}$ on $\overline{\BF}_p\otimes A''$ is a proper linear subspace of $\overline{\BF}_p\otimes A''$. Since $\overline{\BF}_p$ is an infinite field, $\overline{\BF}_p\otimes A''$ cannot be covered by a finite number of proper linear subspaces, so there exists a vector $v\in \overline{\BF}_p\otimes A''$ such that the $K$-orbit of $v$ consists of $|K|$ distinct points, and we can also choose a linear function $\ell\in \Hom(\overline{\BF}_p\otimes A'',\overline{\BF}_p)$ which takes different values on the points of the orbit $Kv$. Then the Lagrange-type polynomial function $P\in \bigoplus_{l=0} ^{|K|-1}S^l(\Hom(\overline{\BF}_p\otimes A'',\overline{\BF}_p))$ defined by
	\[P(x)=\prod_{g\in K\setminus\{e\}}\frac{\ell(x)-\ell(gv)}{\ell(v)-\ell(gv)}\]
	takes the value $1$ at $v$ and vanishes at all other points of the orbit $Kv$. Consequently, for any $g\in K$, the polynomial $gP$ takes the value $1$ at $gv$ and vanishes at all other points of $Kv$. This implies that the representation of $K$ induced on the $K$-submodule of $\bigoplus_{l=0} ^{|K|-1}S^l(\Hom(\overline{\BF}_p\otimes A'',\overline{\BF}_p))$ generated by $P$ is isomorphic to the regular representation of $K$  over $\overline{\BF}_p$. This completes the proof of our claim since all of the irreducible representations $\chi_1,\dots,\chi_c$ are contained in the  regular representation of $K$ over $\overline{\BF}_p$. 
	
	(4) By the Poincar\'e duality, there is a smallest index $\nu$
	such that $H_c^\nu(M;\Orientation)\neq0$.
	By \fref{prop:cohomology-of-elementary-abelian-p-groups}~\fref{item:12} and \fref{item:14}, 
	$\BE_2^{i,\nu}\neq0$ for all $i\ge0$,
	and $\BE_2^{i,j}=0$ for all $j<\nu$, \ $i\ge0$.
	Thus, we have $\BE_{t+1}^{i,\nu}\cong \BE_t^{i,\nu}/d_t(\BE_t^{i-t,\nu+t-1})$  for all $t\ge2$. This implies that if an irreducible representation of $K$ is contained in the module $\overline{\BF}_p\otimes \BE_t^{i,\nu}$, then it is either contained also in $\overline{\BF}_p\otimes \BE_{t+1}^{i,\nu}$ or it must be contained in $\overline{\BF}_p\otimes \BE_t^{i-t,\nu+t-1}$.
	
	If $i>d-\nu$, then $H_c^{i+\nu}(M/A'';\Orientation)=0$ implies $\BE_{\infty}^{i,\nu}=0$, so $\bigcap_{t=2}^{\infty}R_t^{i,\nu}=\emptyset$. By the previous observation, this means that \[R^i=R_2^{i,\nu}\subseteq \bigcup_{t=2}^{i}R_t^{i-t,\nu+t-1}\subseteq \tilde R^{i-2},\]
	therefore $\tilde R^{i}=\tilde R^{i-1}$.
	Iterating this equality, we obtain that for any $i>d-\nu$, we have $\tilde R^{i}=\tilde R^{d-\nu}$. Thus, (3) implies that the set $\tilde R^{d-\nu}\supseteq \tilde R^{2(|K|-1)}$ contains all irreducible representations of $K$. 
	
	(5) Since all the irreducible representations of $K$ are contained in the representation of $K$ on the module $\otimes_{i=0}^{d-\nu}\overline{\BF}_p\otimes H^i(A'',\BF_p)$, 
	\[
          \sum_{i=1}^c d_i\leq \sum_{i=1}^{d-\nu}\dim H^i(A'',\BF_p) \le
          \sum_{i=1}^{d-\nu}\binom{i+r-1}{r-1}=\binom{d-\nu+r-1}{r},
	\]
	hence $|K|\leq \binom{d+r}{r}^2$. 
\end{proof}

\begin{lem} \label{lem:nilpotent-by-cyclic-free-action}
  For all integers $r,d,B\ge0$, there is an integer $n_2(r,d,B)$
  with the following property.\\
  Let $G=P\semidirect H$ be the semidirect product of a $p$-group $P$
  of order at most $p^r$ and a finite group $H$, the order of which is
  not divisible by $p$.
  Let $M$ be a connected $d$-dimensional topological manifold with an effective $G$-action.
  Suppose that
  $\dim H_c^*(M;\Orientation)\le B$,
  the $P$-action on $M$ is free,
  and the induced $H$-action on $H_c^*(M;\Orientation)$ is trivial.
  Then  $H$ has a subgroup $\tilde H$ of index at most $n_2(r,d,B)$
  commuting with $P$.
\end{lem}

\begin{proof}
	We prove the lemma  via induction on $r$.
	For $r=0$, we have $P=\{1\}$, so the lemma holds in this case with $n_2(0,d,B)=1$.
	Assume that $r\ge1$, and $P\neq\{1\}$.
	
	Let $A$ denote the socle of $\CZ(P)$.
	It is a characteristic subgroup of $G$ acting freely on $M$,
	and $A\isom\BZ_p^\rho$ for some positive $\rho\le r$.
	\fref{lem:elementary-p-by-cyclic-free-action}
	gives us a subgroup
	$H_1\le H$
	of index at most
	$\max\big\{n_1(\rho,d) \;\big|\;\allowbreak 0\le\rho\le r \big\}$
	commuting with $A$.
	
	Consider the subgroup $\tilde G=P\semidirect H_1$.
	$A$ is central (hence normal) in $\tilde G$.
	Therefore the quotient space $N=M/A$ is a $d$-dimensional
        topological manifold with an induced action of $\tilde G/A$.
	$H_c^*(N;\Orientation)$ can be calculated via
	the Borel spectral sequence
	(see \fref{prop:Leray-spectral-sequence}).
	Since $A$ is central in $\tilde G$,
	the $H_1$-action on the $\BE_2$ page of the spectral sequence is
	trivial.
	Therefore the $H_1$-action on $H_c^*(N;\Orientation)$ is trivial.
	Moreover, we obtain the bound
	$$
	\dim H_c^*(N;\Orientation) =\dim H_{A,c}^*(M;\Orientation)\le
	\sum_{i+j\le d}\dim \BE_2^{i,j} \le
	\binom{d+r}{r}B.
	$$
	
	The group $\tilde G/A=(P/A)\semidirect H_1$ acts on $N$.
	If $h\in \tilde G$ acts trivially on $N$,
	then there is a unique continuous map $a\colon M\to A$ such
	that $h(x)=a(x)x$. As $A$ is discrete, $M$ is connected, $a(x)\equiv
	a\in A$ is constant. Then $a^{-1}h$ acts trivially on $M$, so
	$h=a\in A$. This implies that $\tilde G/A$ acts effectively on $N$,
	so we can apply the induction hypothesis.
	We obtain a subgroup $\tilde H\le H_1$
	of index at most $n_2\big(r-1,d, \binom{d+r}{r}B\big)$
	acting trivially on $P/A$.
	Recall that $\tilde H$ acts trivially on $A$ as well.
	
	The conjugation action of $\tilde H$ lies
	in the kernel of the restriction homomorphism
	$\Aut(P)\to\Aut(A)\times\Aut(P/A)$.
	This kernel is a $p$-group,
	and $|\tilde H|$ is not divisible by $p$.
	Hence $\tilde H$ acts trivially on $P$.
	Moreover, $|H\colon\tilde H|=|H\colon H_1|\cdot|H_1\colon\tilde H|$
	is bounded in terms of $r,d,B$.
	The induction step is complete.
\end{proof}

\section{Reduction to free action}
\label{sec:Reduction-to-free-action}

The main goal of this section is to prove the following.

\begin{lem} \label{lem:action-of-H-on-cohomology-of-U-is-trivial}
  For all integers $d,B\ge0$, there is an integer $n_3(d,B)$
  with the following property.\\
  Let $p$ be a prime larger than $\tilde C(d,B)$ (defined in \fref{prop:number-of-stapilizers-is-bounded}),
  and let $G=P\semidirect H$ be the semidirect product of a $p$-group $P$
  and a finite group $H$ of order not divisible by $p$.
  Suppose that $G$ acts effectively on
  a connected topological $d$-manifold $M$
  such that
  $\dim H_c^*(M;\Orientation)\le B$,
  and the induced $H$-action on $H_c^*(M;\Orientation)$ is trivial.
  Let $U\subseteq M$ be the largest open subset where the $P$-action
  is free.
  Then $U$ is $G$-invariant, connected,
  \begin{equation}\label{eq:lemma_6.1_ineq}
    \dim H_c^*(U;\Orientation)\le 2^{|\Stab(P,M)|}B\le 2^{\tilde C(d,B)}B,
  \end{equation}
  and $H$ has a subgroup $\tilde H$ of index at most $n_3(d,B)$
  which acts trivially on $H_c^*(U;\Orientation)$.
\end{lem}
\begin{rem}
  A variant of the above inequality is proved in \cite{CMPS}.
  However, we prove the rest of this statement by induction,
  and we need this slightly stronger inequality to make the induction
  work.
\end{rem}

For differentiable action of a finite group on a manifold,
each fixed point has invariant neighborhoods homeomorphic to a ball.
For continuous actions we use the following replacement.

\begin{defn} [{\cite[Definition~I-4.4]{borel1960seminar}}]
  \label{def:adapted-neighbourhood}
  Let $M$ be a topological $d$-manifold
  and $U\subseteq M$ a connected orientable open subset.
  A connected open subset $V\subseteq U$ is \emph{adapted to $U$}
  if the induced homomorphism
  $H^i(V;\Orientation)\to H^i(U;\Orientation)$
  is an isomorphism for $i=d$, and zero for $i\ne d$.
\end{defn}

\begin{lem} \label{lem:invariant-adapted-neighbourhoods-exists}
  Let $G$ be a finite group, $M$ a $d$-dimensional $G$-manifold,
  and $x\in M^G$ a fixed point.
  Then each orientable open neighborhood $U$ of $x$
  contains a $G$-invariant orientable connected open neighborhood of $x$
  adapted to $U$.
\end{lem}
\begin{proof}
  Let $W\subseteq U$ be an open neighborhood of $x$ homeomorphic to $\BR^d$.
  Then every connected open subset in $W$ is adapted to $U$.
  Let $V$ be the intersection of all $G$-translates of $W$,
  and $V^0\subseteq V$ the connected component containing $x$.
  It is $G$-invariant, orientable, connected, and it is adapted to $U$.
\end{proof}

Borel in \cite[Lemma~V-2.1]{borel1960seminar} identified a direct
summand of the equivariant cohomology of a connected oriented
$d$-manifold $M$ which is originated from $H_c^d(M;\BF_p)$.
We need a slightly more precise information than he stated,
and we need to generalize this to non-oriented manifolds.

\begin{prop}\label{prop:edge-like-homomorphism}
  Let $G$ be a finite group and $M$ a $d$-dimensional $G$-manifold.
  The spectral sequence in \fref{prop:Leray-spectral-sequence}
  gives us an edge homomorphism
  \[
    \CE_M\colon
    H_{G,c}^*(M;\Orientation) \longrightarrow
    H^{*-d}(G; H_c^d(M;\Orientation)),
  \]
  which is a natural transformation with respect to
  $G$-equivariant open embeddings.
  If $M$ is connected and the $G$-action has a fixed point, then 
  \begin{equation}\label{eq:isomorphism}
  H^{*}(G; H_c^d(M;\Orientation))\isom H^{*}(G;\BF_p)\otimes H_c^d(M;\Orientation),
  \end{equation} 
  and $\CE_M$ is surjective.
\end{prop}
\begin{proof} By \fref{prop:Leray-spectral-sequence}, we have 
$
	\BE_2^{*,d}= H^*\big(G;H_c^d(M;\Orientation)\big)
$,
	and  $\BE_2^{*,j}=0$ for all $j>d$.	This implies that $\CE_M$ can be defined as the composition
	\[ H_{G,c}^*(M;\Orientation) \longrightarrow 	\BE_{\infty}^{*-d,d} \longrightarrow \BE_{2}^{*-d,d} =
	H^{*-d}(G; H_c^d(M;\Orientation)).\]
	Naturality of $\CE_M$ follows from the naturality of the spectral sequence.
  
  If $M$ is connected, then the $G$-action on $H_c^d(M;\Orientation)$ is trivial by \fref{prop:generalised-Poincare-duality}, therefore \fref{eq:isomorphism} holds.

  Now assume that $M$ is connected and $M^G\neq \emptyset$. Let $x\in M^G$ be a fixed point and $U\subseteq M$
  a connected orientable open neighborhood of $x$.
  By repeated use of \fref{lem:invariant-adapted-neighbourhoods-exists}, construct connected orientable open $G$-invariant neighborhoods
  $U=U_0\supseteq U_1\supseteq\dots\supseteq U_{2d}$ of $x$
  such that each $U_i$ is adapted to $U_{i-1}$, and apply \cite[Lemma~V-2.1]{borel1960seminar} to this sequence.
  The lemma  states that
  there is a subspace $W\le H_{G,c}^*(U;\Orientation)$
  isomorphic to $H^{*-d}(G;\BF_p)\otimes H_c^d(U;\Orientation)$
  and it follows from its proof that $\CE_U$ maps $W$
  isomorphically onto $H^{*-d}(G;\BF_p)\otimes H_c^d(U;\Orientation)$.
  In particular, $\CE_U$ is surjective.
  Consider now the diagram
  \[
    \xymatrix@C50pt{
      H_{G,c}^*(M;\Orientation) \ar[r]_(.4){\CE_M}&
      H^{*-d}(G;\BF_p)\otimes H_c^d(M;\Orientation)
      \\
      \llap{$W\le\ $}
      H_{G,c}^*(U;\Orientation) \ar@{ -{>>}}[r]_(.4){\CE_U}
      \ar[u]^\phi&
      H^{*-d}(G;\BF_p)\otimes H_c^d(U;\Orientation),
      \ar[u]^\psi_\isom
    }
  \]
  where $\phi$ and $\psi$ are the natural homomorphisms induced by
  the inclusion $U\subseteq M$. The map 
  $\psi$ is an isomorphism by \fref{prop:Hd-cohomology-manifold},
  hence $\CE_M\circ\phi=\psi\circ\CE_U$ is surjective.
  This implies that $\CE_M$ is surjective.
\end{proof}

We consider the actions of the group $G\isom\BZ_p$.
Let $M$ be a $G$-manifold.
For differentiable manifolds and smooth actions,
the $G$-action on the normal bundle
of $M^G$ would give us useful information
(see e.g. \cite[Lemma 3.2]{Riera_spheres}).
In our situation the normal bundle is not available.
Instead, we use
the homomorphisms $\Omega_M^k$ defined below.

\begin{lem} \label{lem:Omega_M^k}
  Consider the group $G\isom\BZ_p$.
  Let $M$ be a $d$-dimensional $G$-manifold,
  and suppose that $F=\Fix{M}{G}$ is non-empty.
  For each $n\ge1$ and each $k\le d$, we have:
  \begin{enumerate}[\indent(a)]
  \item \label{item:17}
    The restriction homomorphism
    $H_{G,c}^{d+n}(M;\Orientation)
    \stackrel{\res_M}\longrightarrow
    H_{G,c}^{d+n}(F;\Orientation)$
    is an isomorphism.
  \item \label{item:18}
    In the following diagram $\CK_F^k$ is the natural inclusion
    induced by the K\"unneth isomorphism
    (see \fref{prop:free-action-cohomology}.\fref{item:9}),
    $\CE_M$ is the edge homomorphism
    defined in \fref{prop:edge-like-homomorphism},
    and $\Omega_M^k=\CE_M\circ\res_M^{-1}\circ\CK_F^k$.
    $$
    \xymatrix@C48pt{
      H^{n+d-k}(G;\BF_p)\otimes H_c^k(F;\Orientation)
      \ar@{^(->}[d]^{\CK_F^k}
      \ar@/_25pt/@{.>}[r]^{\Omega_{M}^k}
      &
      H^{n}(G; H_c^d(M;\Orientation))
      \\
      H_{G,c}^{n+d}(F;\Orientation)
      \ar[r]^{\res_M^{-1}}
      &
      H_{G,c}^{n+d}(M;\Orientation)
      \ar[u]^{\CE_M}
    }
    $$
    The diagram depends on $M$ functorially
    with respect to
    $G$-equi\-vari\-ant open embeddings.
  \item \label{item:19}
    Suppose that $F_0$ is a
    $k$-dimensional non-empty connected component of $F$,
    and let $M_0$ denote the connected component of $M$ containing
    $F_0$.
    Then $\Omega_M^k$ maps the direct summand
    $$
    H^{n+d-k}(G;\BF_p)\otimes H_c^k(F_0;\Orientation)\le
    H^{n+d-k}(G;\BF_p)\otimes H_c^k(F;\Orientation)
    $$
    isomorphically onto the direct summand
    $$
    H^{n}(G;\BF_p)\otimes H_c^d(M_0;\Orientation)\le
    H^{n}(G; H_c^d(M;\Orientation)).
    $$
  \end{enumerate}
\end{lem}
\begin{rem}\label{rem:c-makes-sense}
  We show that
  \fref{item:19} is meaningful.
  $G$ maps $M_0$ into itself, hence
  $H^n\big(G;H_c^d(M_0;\Orientation)\big)$ is a direct summand of
  $H^{n}(G; H_c^d(M;\Orientation))$.
  Moreover, by \fref{prop:Hd-cohomology-manifold}
  the $G$-action on $H_c^d(M_0;\Orientation)$ is trivial, so
  $H^n\big(G;H_c^d(M_0;\Orientation)\big)\isom
  H^{n}(G;\BF_p)\otimes H_c^d(M_0;\Orientation)$.
  This shows that $H^{n}(G;\BF_p)\otimes H_c^d(M_0;\Orientation)$
  is indeed a direct summand of
  $H^{n}(G; H_c^d(M;\Orientation))$.
  The same argument shows that
  $H^{n}(G;\BF_p)\otimes H_c^d(F_0;\Orientation)$
  is a direct summand of
  $H^{n}(G; H_c^d(F;\Orientation))$.
\end{rem}
\begin{proof}[Proof of \fref{lem:Omega_M^k}]
  Let $M_*\subseteq M$ be the $G$-orbit of a connected component of
  $M$. If $M_*\cap F=\emptyset$, then
  $H_{G,c}^{n+d}(M_*;\Orientation)=0$
  by \fref{prop:free-action-cohomology}\fref{item:8},
  and so $\res_{M_*}=0$.
  Otherwise
  $M_*$ is connected and $M_*\setminus F$ is open, hence 
  $H_{G,c}^{n+d+1}(M_*\setminus F;\Orientation)=
  H_{G,c}^{n+d}(M_*\setminus F;\Orientation)=0$
  by \fref{prop:free-action-cohomology}\fref{item:8},
  so
  \fref{prop:long-exact-sequence-of-closed-subset}
  implies that $\res_{M_*}$ is an isomorphism.
  Since $\res_M$ is the direct sum of the isomorphisms $\res_{M_*}$
  for various $M_*$, \fref{item:17} follows.

  Propositions~\ref{prop:free-action-cohomology}\fref{item:9},
  \ref{prop:long-exact-sequence-of-closed-subset},
  and \ref{prop:edge-like-homomorphism}
  state that $\CK_F^k$, $\res_M$, and $\CE_M$
  are natural transformations.
  This proves \fref{item:18}.

  Finally we prove \fref{item:19}.
  By \fref{prop:fixed-point-set-is-manifold} $F_0$ is a cohomology manifold.
  Let $V\subseteq F_0$ be a connected orientable open subset,
  and let $V'\subset V$ be an open subset such that
  the induced homomorphism
  \begin{equation}
    \label{eq:2}
    H_c^j(V';\Orientation)\to H_c^j(V;\Orientation)
    \text{ \ is }
    \begin{cases}
      \text{an isomorphism for \ } j=k, \\
      0 \text{ \ \ otherwise.}
    \end{cases}    
  \end{equation}
  Let $W\subseteq M$ be an open subset such that $W\cap F = V$,
  and let $U$ be the connected component of $\bigcap_{g\in G}gW$
  containing $V$.
  Similarly, let $W'\subseteq W$ be an open subset
  such that $W'\cap F = V'$,
  and let $U'$ be the connected component of $\bigcap_{g\in G}gW'$
  containing $V'$.
  Then $U'\subseteq U$ are $G$-invariant connected open submanifolds,
  hence the inclusion map induces an isomorphism
  \[
    H_c^d(U';\Orientation)
    \stackrel\isom\longrightarrow
    H_c^d(U;\Orientation)\isom\BF_p.
  \]
  Applying \fref{item:17}, \fref{item:18}, and
  \fref{rem:c-makes-sense}
  to $U$ and $U'$, we obtain the following diagram.
  \[
    \xymatrix@C48pt{
      H_{G,c}^{n+d}(V';\Orientation)
      \ar[d]^\alpha
      \ar@(l,l)@<-25pt>[ddd]^{\res_{U'}^{-1}}_\isom
      &
      H^{n+d-k}(G;\BF_p)\otimes H_c^k(V';\Orientation)
      \ar@{_(->}[l]_(.6){\CK_{V'}^k}
      \ar[d]^\isom
      \\
      H_{G,c}^{n+d}(V;\Orientation)
      \ar[d]^{\res_U^{-1}}_\isom
      &
      H^{n+d-k}(G;\BF_p)\otimes H_c^k(V;\Orientation)
      \ar@{_(->}[l]_(.6){\CK_V^k}
      \ar@/_10pt/@{ .>}[d]^{\Omega_{U}^k}
      \\
      H_{G,c}^{n+d}(U;\Orientation)
      \ar@{ -{>>}}[r]^(.4){\CE_U}
      &
      H^{n}(G;\BF_p)\otimes H_c^d(U;\Orientation)
      \\
      H_{G,c}^{n+d}(U';\Orientation)
      \ar@{ ->>}[r]^(.4){\CE_{U'}}
      \ar[u]_\beta
      &
      H^{n}(G;\BF_p)\otimes H_c^d(U';\Orientation)
      \ar[u]^\gamma_\isom
    }
  \]
  \fref{prop:free-action-cohomology}\fref{item:9} and \fref{eq:2} imply that
  \[
    \im(\CK_V^k)=\im(\alpha),
  \]
  hence
  \[
    \im(\res_U^{-1}\circ\CK_V^k)=\im(\res_U^{-1}\circ\alpha)=\im(\beta),
  \]
  and therefore
  \[
    \im(\Omega_{U}^k) = \im(\CE_U\circ\res_U^{-1}\circ\CK_V^k) =
    \im(\CE_U\circ\beta) = \im(\gamma\circ\CE_{U'}).
  \]
  But $\CE_{U'}$ is surjective by \fref{prop:edge-like-homomorphism},
  hence $\Omega_U^k$ is surjective as well.
  Now $\Omega_U^k$ is a surjective homomorphism between
  one-dimensional spaces, so it is an isomorphism.

  In the following diagram, the vertical arrows
  are induced by the inclusions $V\subseteq F_0\subseteq F$
  and $U\subseteq M_0\subseteq M$.
  \[
    \xymatrix@C28pt{
      H^{n+d-k}(G;\BF_p)\otimes H_c^k(F;\Orientation)
      \ar[r]^(.4){\Omega_{M}^k}
      &
      H^{n}(G; H_c^d(M;\Orientation))
      \isom
      H^{n}(G;\BF_p)\otimes H_c^d(M;\Orientation)
      \\
      H^{n+d-k}(G;\BF_p)\otimes H_c^k(F_0;\Orientation)
      \ar@{^{(}->}[u]
      &
      H^{n}(G;\BF_p)\otimes H_c^d(M_0;\Orientation)
      \ar@{^{(}->}[u]
      \\
      H^{n+d-k}(G;\BF_p)\otimes H_c^k(V;\Orientation)
      \ar[u]^\phi
      \ar[r]^{\Omega_{U}^k}_\isom
      &
      H^{n}(G;\BF_p)\otimes H_c^d(U;\Orientation).
      \ar[u]^\psi
    }
  \]
  By \fref{prop:Hd-cohomology-manifold}, $\phi$ and $\psi$ are isomorphisms.
  This proves \fref{item:19}.
\end{proof}

\begin{defn}
  Let $W$ and $I$ be finite dimensional ${\BF}H$-modules for some field $\BF$ and a finite group $H$, and assume that $I$ is irreducible. We shall denote by $\#_IW$ the number of occurrences of $I$ among the composition factors of $W$.
\end{defn}

Smith theory can be used to study the topology of the fixed point set
of a $p$-group acting on a manifold
(see e.g. \fref{prop:fixed-point-set-is-manifold}). 
We need the following generalization when the action of another group
$H$ is taken into account.

\begin{lem}\label{lem:Smith-theory-generalised}
  Let $G=P\times H$ be the direct product of a finite $p$-group $P$
  for some prime $p$ and a finite group  $H$.
  Let $M$ be a $G$-manifold.
  Then $\Fix{M}{P}$ is an $H$-invariant union of disjoint cohomology manifolds over $\BF_p$,
  and for any irreducible $H$-module $I$,
  we have
  $$
  \#_IH_c^*\big(\Fix{M}{P};\Orientation\big) \le
  \#_IH_c^*(M;\Orientation),
  $$
  $$
  \#_IH_c^*\big(M\setminus\Fix{M}{P};\Orientation\big) \le
  2\,\#_IH_c^*(M;\Orientation).
  $$
\end{lem}
\begin{proof}
  We prove the first inequality by induction on $|P|$.
  Assume first, that $P=\BZ_p$.
  We recall the proof of
  \cite[Theorem II-19.7]{bredon2012sheaf},
  and adjust it to our needs.

  Let $\pi\colon M\to M/P$ denote the orbit map,
  $F=\Fix{M}{P}$, and $\iota\colon F\hookrightarrow M/P$ denote the
  restriction of $\pi$ to $F$, which is a homeomorphic embedding.
  Then $H$ acts naturally on $M$, $F$, $M/P$,
  and the maps $\pi$, $\iota$ are $H$-equivariant.
  Therefore, by \fref{prop:fixed-point-set-is-manifold},
  $\Fix{M}{P}$ is an $H$-invariant disjoint union of cohomology manifolds over $\BF_p$.
  Let \ $\res\colon \pi_*\Orientation[M]\to\iota_*\Orientation[F]$
  denote the restriction homomorphism.
  
  We choose a generator $g\in P$, and set $\tau=1-g$,
  $\sigma=1+g+g^2+\dots + g^{p-1}$ in the group ring $\BF_pP$.
  For sheaves of $\BF_pP$-modules,
  we denote by $\tilde\sigma$ and $\tilde\tau$ the automorphisms
  given by the multiplications with $\sigma$ and $\tau$.
  Then \cite[Theorem II-19.7]{bredon2012sheaf} gives us
  the following two exact sequences of sheaves on $M/P$:
  $$
  0\to
  \tilde\sigma(\pi_*\Orientation[M]) \hookrightarrow
  \pi_*\Orientation[M] \stackrel{\tilde\tau\;\oplus\;\res}\longrightarrow
  \tilde\tau(\pi_*\Orientation[M])\oplus\iota_*\Orientation[F] \to0
  $$  
  $$
  0\to
  \tilde\tau(\pi_*\Orientation[M]) \hookrightarrow
  \pi_*\Orientation[M] \stackrel{\tilde\sigma\;\oplus\;\res}\longrightarrow
  \tilde\sigma(\pi_*\Orientation[M])\oplus\iota_*\Orientation[F] \to0
  $$  
  Since $H$ and $g$ commute, $H$ acts canonically on these sheaves,
  and the homomorphisms are $H$-equivariant.
  Therefore, the corresponding long exact sequences
  are sequences of $\BF_pH$-modules and module homomorphisms.
The parts
\begin{align*}
  H_c^{2k}\big(M/P;\pi_*\Orientation[M]\big) \to
  H_c^{2k}\big(M/P;\tilde\tau(\pi_*\Orientation[M])\big) \oplus
  H_c^{2k}\big(F;\Orientation[F]\big) \to&\\ \to
  H_c^{2k+1}\big(M/P;\tilde\sigma(\pi_*\Orientation[M])\big)\to\cdots&
\end{align*}
and
\begin{align*}
H_c^{2k+1}\big(M/P;\pi_*\Orientation[M]\big)\!\! \to \!\!
  H_c^{2k+1}\big(M/P;\tilde\sigma(\pi_*\Orientation[M])\big) \oplus
  H_c^{2k+1}\big(F;\Orientation[F]\big) \!\!\to&\\ \to
  H_c^{2k+2}\big(M/P;\tilde\tau(\pi_*\Orientation[M])\big)\to\cdots&
\end{align*}
 of these long exact sequences imply the inequalities
\begin{align*}
  \#_IH_c^{2k}\big(&M/P;\tilde\tau(\pi_*\Orientation[M])\big) +
  \#_IH_c^{2k}\big(F;\Orientation[F]\big) \\&\le
  \#_IH_c^{2k}\big(M/P;\pi_*\Orientation[M]\big) +
  \#_IH_c^{2k+1}\big(M/P;\tilde\sigma(\pi_*\Orientation[M])\big)
\end{align*}
and
 \begin{align*}
  \#_IH_c^{2k+1}\big(&M/P;\tilde\sigma(\pi_*\Orientation[M])\big) +
  \#_IH_c^{2k+1}\big(F;\Orientation[F]\big)\\& \le
  \#_IH_c^{2k+1}\big(M/P;\pi_*\Orientation[M]\big) +
  \#_IH_c^{2k+2}\big(M/P;\tilde\tau(\pi_*\Orientation[M])\big)
  \end{align*}
 for all $k\ge 0$. 
  Adding these inequalities up for all $k\ge0$, and cancelling the repeated
  terms, we obtain
  $$
  \#_IH_c^*\big(F;\Orientation[F]\big) \le
  \#_IH_c^*(M/P;\pi_*\Orientation[M]).
  $$
  The fibers of $\pi$ are finite,
  hence the Leray spectral sequence of $\pi$
  (see \cite[Theorem IV-6.1]{bredon2012sheaf})
  degenerates, and so
  $H_c^*(M/P;\pi_*\Orientation[M])\isom
  H_c^*(M;\Orientation[M])$.
  This proves the first inequality of the lemma.
  
  The second inequality follows from the first one
  by \fref{prop:fixed-point-set-is-manifold},
  and the long exact sequence for cohomology with compact support
  (see \cite[II-10.3]{bredon2012sheaf}).

  Next we do the induction step.
  We choose a subgroup $A\isom\BZ_p$ in the center of $G$.
  Then $N=\Fix{M}{A}$
  is a $G$-invariant union of disjoint cohomology manifolds over $\BF_p$,
  and $G/A$ acts on $N$ canonically.
  Applying the induction hypothesis first to the $A$-action on $M$,
  and then to the $G/A$-action on $N$,
  we see that our statements hold in this case.
  This completes the induction step.
\end{proof}

Now we are ready to prove the following special case of
\fref{lem:action-of-H-on-cohomology-of-U-is-trivial}. 

\begin{lem}\label{lem:Zp-by-cyclic-fixed-point}
  For all integers $d,B\ge0$, there is an integer $n_4(d,B)$
  with the following property.
  Let $p$ be a prime, and
  $G$ be a finite group of the form $G=A\semidirect H$,
  where $A$ is isomorphic to $\BZ_p$ equipped with an $H$-action,
  and $|H|$ is not divisible by $p$.
  Let $M$ be a $d$-dimensional effective $G$-manifold over $\BF_p$
  such that
  $\dim H_c^*(M;\Orientation)\le B$.
  Suppose that $F=\Fix{M}{A}$ is nonempty,
  and the induced $H$-action on $H_c^*(M;\Orientation)$ is trivial.
  Then $H$ has a subgroup $\tilde H$ of index at most $n_4(d,B)$
  such that
  \begin{enumerate}[\indent(a)]
  \item \label{item:24}
    $\tilde H$ commutes with $A$, and
  \item \label{item:25}
    $\tilde H$ acts trivially
    on $H_c^*(F;\Orientation)$
    and on $H_c^*\big(M\setminus F;\Orientation\big)$.
  \end{enumerate}
\end{lem}
\begin{proof}
  If $p=2$, then $A\isom\BZ_2$ has no nontrivial  automorphism,
  so \fref{item:24} holds with $\tilde H=H$.

  Assume now that $p>2$.
  We identify $\Aut(A)$ with the multiplicative group $\BF_p^*$,
  then $H$ acts on $A$ via a character
  $\lambda\colon H\to\BF_p^*$.
  \fref{prop:cohomology-of-elementary-abelian-p-groups} implies that
  $H$ acts on $H^{n}(A;\BF_p)\isom\BF_p$
  via the character $\lambda^{-\lceil\frac n2\rceil}$ for all
  $n\ge0$. 

  $H$ permutes the connected components of $F$ and $M$.
  Let $F_0$ be a connected component of $F$ of dimension say $k$,
  and let $M_0$ be the connected component of $M$ containing $F_0$.
  \fref{prop:fixed-point-set-is-manifold}
  and \fref{prop:Hd-cohomology-manifold}
  imply that $F$ has at most $B$ connected components,
  so $H$ has a subgroup $H_0$ of index at most $B$
  that maps $F_0$ into itself,
  and therefore it maps $M_0$ also into itself.
  By \fref{prop:Hd-cohomology-manifold},
  the $H_0$-action on $H_c^k(F_0;\Orientation)$ and
  $H_c^d(M_0;\Orientation)$
  is trivial.

  Choose an even integer $n>0$.
  \fref{lem:Omega_M^k} gives us
  an $H_0$-equivariant  isomorphism
  $$
  H^{n+d-k}(A;\BF_p)\otimes H_c^k(F_0;\Orientation)
  \stackrel\isom\longrightarrow
  H^{n}(A;\BF_p)\otimes H_c^d(M_0;\Orientation).
  $$
  Let $\lambda_0$ denote the restriction of $\lambda$ to $H_0$.
  Since $H_0$ acts on the two sides via the characters
  $\lambda_0^{-\lceil\frac{n+d-k}2\rceil}$  and
  $\lambda_0^{-\frac{n}2}$
  respectively,
  we obtain that $\lambda_0^{\lceil\frac{d-k}2\rceil}=1$.
  Since $k<d$,
  the subgroup $\tilde H=\ker(\lambda_0)$ has index at most
  $\lceil\frac{d-k}2\rceil$ in $H_0$,
  so it has index at most $B\lceil\frac{d-k}2\rceil$ in $H$.
  Moreover, $\ker(\lambda_0)$ commutes with $A$.
  This implies \fref{item:24} for odd primes as well.

  In particular, $A\times\tilde H$ is a subgroup of $G$.
  We consider $M$ as an $(A\times\tilde H)$-manifold over
  $\BF_p$.
  Then $H_c^*(M;\Orientation)$ is an $\tilde H$-module
  with no non-trivial composition factors.
  \fref{lem:Smith-theory-generalised} implies that
  $H_c^*(F;\Orientation)$ and
  $H_c^*\big(M\setminus F;\Orientation\big)$
  are also $\tilde H$-modules
  with no non-trivial composition factors.
  This proves \fref{item:25}.
\end{proof}

\begin{lem} \label{lem:Cohomologically-free-action-on-complement-of-fixedpoint-set}
  For all integers $d,B\ge0$, there is an integer $n_5(d,B)$
  with the following property.\\
  Let $p$ be a prime larger than the bound $\tilde C(d,B)$ defined in \fref{prop:number-of-stapilizers-is-bounded},
  and let $G=P\semidirect H$ be the semidirect product of a $p$-group $P$
  and a finite group $H$ of order not divisible by $p$.
  Let $M$ be a connected topological $d$-manifold
  with an effective $G$-action.
  Suppose that
  $\dim H_c^*(M;\Orientation)\le B$,
  and the induced $H$-action on $H_c^*(M;\Orientation)$ is trivial.
  Let $K\in\Stab(P,M)$ be a stabilizer subgroup different from $\{1\}$.
  Then
  $H$ has a subgroup $H_1$ of index at most $n_5(d,B)$
  and $K$ has a subgroup $L\ne\{1\}$
  such that $M\setminus M^L$ is $H_1$-invariant
  and the induced $H_1$-action on
  $H_c^*(M\setminus M^L;\Orientation)$
  is trivial.
\end{lem}

\begin{proof}
  Proper subgroups of $P$ have index at least $p$,
  so  by \fref{prop:number-of-stapilizers-is-bounded}
  \[
    \big|\Stab(P,M)\big|\le\tilde C(d,B).
  \]
  Conjugates of $K$ are also stabilizer subgroups,
  so $K$ has at most $\tilde C(d,B)$ conjugates.
  Therefore the normalizer subgroup $\CN_H(K)$ has index at most
  $\tilde C(d,B)$ in $H$.
  
  Let $E$ be the socle of $K$.
  This is a non-trivial elementary abelian $p$-subgroup. If $E$ acts
  freely on $M$, then
  $H_1=\CN_H(K)$ and $L=E$ are a good choice.
  Otherwise there is a point $x\in M$ with a non-trivial stabilizer
  $E_x$ in $E$.
  If $|E_x|>p$, then Borel's fixed point formula
  (\fref{prop:Borel-fixpoint-formula})
  implies that there is a nearby point $y\in M$ whose stabilizer
  $E_y$ has index $p$ in $E_x$. 
  Iterating this argument we obtain a stabilizer subgroup
  $L\in\Stab(E,M)$ of order $p$.
  
  Since $E$ is characteristic in $K$, it is normalized by $\CN_H(K)$.
  As above, \fref{prop:number-of-stapilizers-is-bounded}
  implies that
  \[
    \big|\Stab(E,M)\big|\le\tilde C(d,B),
  \]
  and $\CN_H(K)$ acts on  $\Stab(E,M)$ via conjugation.
  Hence,  $\CN_H(K)$ has a subgroup $\tilde H\le\CN_H(K)$
  of index at most $\tilde C(d,B)$
  which normalizes $L$, and therefore
  $M^L$ and $M\setminus M^L$ are $\tilde H$-invariant.

  Now we apply \fref{lem:Zp-by-cyclic-fixed-point}
  to the subgroup $L\semidirect \tilde H\le G$
  with its given action on $M$.
  We obtain a subgroup $H_1\le \tilde H$
  with the required properties.
\end{proof}

\begin{proof}[Proof of
  \fref{lem:action-of-H-on-cohomology-of-U-is-trivial}] 
  By construction, $U$ is $P$-invariant, and it is also $H$-invariant
  since $H$ normalizes $P$. Therefore $U$ is $G$-invariant.
  Since $p>\tilde C(d,B)$,
  \fref{prop:number-of-stapilizers-is-bounded} implies the second part
  of inequality \fref{eq:lemma_6.1_ineq}.
  
  We prove the rest of the statement
  by induction on the size of $\Stab(P,M)$.
  If $\big|\Stab(P,M)\big|=1$, then $\{1\}$ is the only stabilizer subgroup,
  hence the $P$-action is free, $U=M$, and the statement holds.
  For the induction step we assume that $\big|\Stab(P,M)\big|>1$,
  and  the statement holds in all cases with a smaller number of
  stabilizer subgroups.
  
  Let $K\in \Stab(P,M)$ be a stabilizer different
  from $\{1\}$.
  \fref{lem:Cohomologically-free-action-on-complement-of-fixedpoint-set}
  gives us a subgroup $H_1\le H$ of index at most $n_5(d,B)$
  and  a non-trivial subgroup $L\le K$
  such that the open submanifold $V=M\setminus M^L$
  is $H_1$-invariant and the $H_1$-action on
  $H_c^*(V;\Orientation)$ is trivial.
  Moreover, \fref{prop:fixed-point-set-is-manifold} implies that $V$ is
  connected and
  \[
    \dim H_c^*(V;\Orientation)\le2B.
  \]  
  Since $M^L$ contains all points whose stabilizer is $K$, we have
  \[
    \Stab(P,V) \subseteq \Stab(P,M)\setminus\{K\}.
  \]
  We may apply the induction hypothesis to the subgroup
  $P\semidirect H_1\le G$ with its given action on $V$.
  This gives us a subgroup $\tilde H\le H_1\le H$ with all the
  required properties.
  The induction step is complete.
\end{proof}

\section{Proof of the Main Theorem}
\label{sec:proof-main-theorem}

\begin{lem}\label{lem:Cohomologically-trivial-actions-Minkowski}
	For all integers $B,\tau\ge 0$, there is an integer $I(B,\tau)$
	with the following property.\\
	Let $G$ be a finite group acting continuously on a topological
	manifold $M$
	such that
        $H_*(M;\BZ)$ has rank $B$
        and has at most $\tau$ torsion elements.
	Then $\dim H_c^*(M;\Orientation)\le B+2\tau$ for all primes $p$.
	Moreover,
	$G$ has a normal subgroup of index at most $I(B,\tau)$
	whose canonical action on $H^*_c(M;\Orientation)$ is trivial for all $p$.
\end{lem}
\begin{proof}
	By the Poincar\'e duality
	(\fref{prop:generalised-Poincare-duality}),
	it is enough to prove that $\dim H_*(M;\BZ_p)\le B+2\tau$,
	and to find a subgroup of bounded index
	which acts trivially on the homology groups
	$H_*(M;\BZ_p)$ for all primes $p$.
	
	Let $\CT$ be the torsion part of $H_*(M;\BZ)$,
	and $\CF=H_*(M;\BZ)/\CT\isom\BZ^B$ be the free part.
	The $G$-action on $M$ induces a canonical $G$-action on  the split
	short exact sequence $0\to \CT\to H_*(M;\BZ) \to\CF\to 0$. This
	action gives a homomorphism of $G$ into the automorphism group of
	this short exact sequence, which is isomorphic to the semidirect
	product
        $\Hom(\CF,\CT)\semidirect\big(\Aut(\CT)\times\Aut(\CF)\big)$.
        The kernel
	$G_0$ of this homomorphism acts trivially on $H_*(M;\BZ)$. 
	A theorem of Minkowski~\cite{Minkowski}
	gives an upper bound $I_0(B)$ on 
	the size of a finite subgroup of 
	$\Aut(\BZ^B)\cong\Aut(\CF)$,
	and 
	$\big|\Aut(\CT)\big|\le|\CT|^{\log |\CT|}$,
	hence the index of  $G_0$ in $G$ is at most $\tau^{\log \tau}\cdot I_0(B)\cdot \tau^B$.
	
	By  the universal coefficient theorem,
	there is a split exact sequence
	\begin{equation}\label{eq:univ_coeff}
	0\to
	H_*(M;\BZ)\otimes \BZ_p \to
	H_*(M;\BZ_p)\to
	\Tor\big(H_*(M;\BZ),\BZ_p\big)\to
	0
	\end{equation}
	for all primes $p$, where
	\begin{equation}
	\label{eq:15}
	\Tor\big(H_*(M;\BZ),\BZ_p\big)=\{h\in H_*(M;\BZ):ph=0\},
	\end{equation}
	which has at most $\tau$ elements.
	Moreover, $H_*(M;\BZ)\otimes \BZ_p$ is an elementary abelian
	$p$-group of rank at most $\tau+B$.
	This implies that $\dim H_*(M;\BZ_p)\le B+2\tau$.
	
	The action of $G$ on $M$ induces a homomorphism of $G$ into
	the automorphism group of the split exact sequence
	\eqref{eq:univ_coeff}, which is isomorphic to the semidirect product
	of  
	$\Aut\big(H_*(M;\BZ)\otimes\BZ_p\big)\times
	\Aut\big(\Tor(H_*(M;\BZ),\BZ_p)\big)$ and the
	normal subgroup
	$\Hom\big(\Tor(H_*(M;\BZ),\BZ_p),H_*(M;\BZ)\otimes\BZ_p\big)$.
	Since $G_0$ acts trivially on $H_*(M;\BZ)$, the
	restriction of this homomorphism onto $G_0$ yields a homomorphism 
	$$
	G_0\to\Hom\big(\Tor(H_*(M;\BZ),\BZ_p),H_*(M;\BZ)\otimes\BZ_p\big),
	$$
	the kernel $G_p$ of which acts trivially on
	$H_*(M;\BZ_p)$. By \eqref{eq:15} the torsion group 
	$\Tor\big(H_*(M;\BZ),\BZ_p\big)$
	is trivial if $p>\tau$ and has at most $\tau$ elements if $p\leq \tau$. The
	tensor product $H_*(M;\BZ)\otimes \BZ_p$ is an elementary abelian
	$p$-group of rank at most $\tau+B$.  Consequently, we have $G_p=G_0$ if
	$p>\tau$, and $|G_0:G_p|\leq p^{(\tau+B)\tau}$ if $p\leq \tau$.
	
	The normal subgroup of those elements of $G$ that act trivially on
	$H_*(M;\BZ_p)$ for every prime $p$ contains the subgroup
	$G_0\cap\bigcap_{p\leq \tau}G_p$, therefore, its index is bounded from
	above by the index of the latter subgroup, which is at most 
	\[
	|G:G_0|\cdot\prod_{p\leq \tau}|G_0:G_p|\leq  \tau^{\log \tau}\cdot I_0(B)\cdot
	\tau^{B+(\tau+B)\tau^2}.\qedhere
	\] 
\end{proof}

Now we are ready to present the  
\begin{proof}[Proof of \fref{thm:Ghys-conjecture-open-manifolds}]
Let $M_1,\dots,M_k$ be the connected components of $M$, $B$ be the
rank of $H_*(M;\BZ)$, $\tau$ the size of the
torsion part of $H_*(M;\BZ)$, and $d = \dim(M)$.
By \fref{prop:Hd-cohomology-manifold} and
\fref{lem:Cohomologically-trivial-actions-Minkowski}, 
we may assume that the canonical $G$-action on $H_c^*(M;\Orientation)$
is trivial, and we have 
\[
  k=\dim H^d_c(M,\Orientation) \leq
  \dim H_c^*(M;\Orientation) \le
  B+2\tau
\]  
for all primes $p$. Then $G$ acts trivially on $H_0(M,\BF_p)$, hence
the elements of $G$ map each component of $M$ into itself, and we
obtain homomorphisms $\phi_i\colon G\to \Homeo(M_i)$ corresponding to
the action of $G$ on $M_i$. Let $G_i$ denote the image $\phi_i(G)$ of
$G$ under $\phi_i$. 

Take any of the indices $1\le i\le k$. \fref{thm:Mann-Su_v0} gives us
a bound $\rank(G_i)\le r=r\big(d,H_*(M;\BZ)\big)$. 
Using the bounds in
\fref{prop:number-of-stapilizers-is-bounded},
\fref{lem:action-of-H-on-cohomology-of-U-is-trivial}, and
\fref{lem:nilpotent-by-cyclic-free-action}
we define the constants
$$
\tilde C = \tilde C(d,B),
\quad\quad
n_3=n_3(d,B),
\quad\quad
 n_2=n_2(2r,d,2^{\tilde C}B).
$$
Apply \fref{thm:Nilpotent-or-Special-by-cyclic_v0} to each of the
groups $G_i$
with the parameter
\[
  T = \max\big(\,(\tilde C^{2r})!\,,\;  n_2 n_3\,\big)+1.
\]
The theorem has two possible outcomes.
Assume first that
\ref{thm:Nilpotent-or-Special-by-cyclic_v0} \fref{item:2} holds,
i.e.
$G_i$ has a subgroup of the form $P\semidirect_\alpha H$,
where $P$ is a $p$-group of order at most $p^{2r}$, $H$ is an abelian group,
and $\alpha\colon H\to\Aut(P)$ is the action of $H$ on $P$,
such that
\[
  p \nmid |H|,
  \quad\quad
  \big|\alpha(H)\big|\ge T.
\]
This implies that $(p^{2r})!> \big|\Aut(P)\big|\ge T>(\tilde C^{2r})!$,
hence $p>\tilde C$.
We apply \fref{lem:action-of-H-on-cohomology-of-U-is-trivial}
to the subgroup $P\semidirect_\alpha H\le G_i$
acting on $M_i$.
We obtain a $(P\semidirect_\alpha H)$-invariant open connected submanifold $U\subseteq M_i$ on which $P$ acts freely such that
\[
  \dim H_c^*(U;\Orientation)\le 2^{\tilde C}B,
\]
and a subgroup $H_1\le H$ of index at most $n_3$
acting trivially on $H_c^*(U;\Orientation)$.
We apply \fref{lem:nilpotent-by-cyclic-free-action}
to the $P\semidirect H_1$ action on $U$,
and obtain a subgroup $\tilde H\le H_1$ of index at most $n_2$
commuting with $P$.
But then $\tilde H\le\ker(\alpha)$, hence
$\big|\im(\alpha)\big|\le|H:H_1|\cdot|H_1:\tilde H|\le n_2n_3<T$,
a contradiction.
	
Thus 
\ref{thm:Nilpotent-or-Special-by-cyclic_v0} \fref{item:1}
must hold, i.e., for each $i$, there is a nilpotent normal subgroup
$N_i$ of index at most $I(r,T)$ in $G_i$. The map $\phi\colon G\to
G_1\times\dots\times G_k$, $\phi(g)=(\phi_1(g),\dots,\phi_k(g))$ is an
injective homomorphism, so
$\phi^{-1}\big(\phi(G)\cap(N_1\times\dots\times N_k)\big)$ is a
nilpotent normal subgroup of index at most $I(r,T)^{B+2\tau}$ in
$G$. This proves the theorem. 
\end{proof}

\bibliographystyle{amsplain}
\bibliography{Ghys}

\noindent Bal\'azs Csik\'os, Institute of Mathematics, E\"otv\"os Lor\'and University, \\Budapest, P\'azm\'any P. stny. 1/C, H-1117 Hungary. 
\\ \emph{E-mail address:} balazs.csikos@ttk.elte.hu

L\'aszl\'o Pyber, Alfr\'ed R\'enyi Institute of Mathematics, Budapest, Re\'altanoda u. 13-15, H-1053 Hungary,\\ \emph{E-mail address:} pyber.laszlo@renyi.mta.hu

Endre Szab\'o, Alfr\'ed R\'enyi Institute of Mathematics, Budapest, Re\'al\-tanoda u. 13-15, H-1053 Hungary,\\ \emph{E-mail address:} szabo.endre@renyi.mta.hu

\end{document}